\documentclass[11pt]{article}

\usepackage[utf8]{inputenc}

\usepackage{geometry}
\usepackage{graphicx}

\usepackage{booktabs} % for much better looking tables
\usepackage{array} % for better arrays (eg matrices) in maths
\usepackage{paralist} % very flexible & customisable lists (eg. enumerate/itemize, etc.)
\usepackage{verbatim} % adds environment for commenting out blocks of text & for better verbatim
\usepackage{subfig} % make it possible to include more than one captioned figure/table in a single float
\usepackage{tabularx}
\usepackage{fancyhdr} % This should be set AFTER setting up the page geometry
\usepackage[nottoc,notlof,notlot]{tocbibind}
\usepackage[titles,subfigure]{tocloft}

\usepackage{amsmath,amsfonts,amsthm,mathrsfs,amssymb,cite}

\newcommand{\R}{{\mathbb R}}
\newcommand{\Z}{{\mathbb Z}}
\newcommand{\N}{{\mathbb N}}
\newcommand{\C}{{\mathbb C}}
\newcommand{\K}{{\mathcal K}}
\newcommand{\U}{{\mathcal U}}
\newcommand{\W}{{\mathcal W}}
\newcommand{\CC}{{\mathcal C}}

\newcommand{\PP}{{\mathbb P}}
\newcommand{\s}{{\mathbb S}}
\newcommand{\F}{{\mathcal F}}

\newcommand{\be}{\begin{eqnarray}}
\newcommand{\ben}{\begin{eqnarray*}}
\newcommand{\en}{\end{eqnarray}}
\newcommand{\enn}{\end{eqnarray*}}

\newcommand{\ima}{{\rm Im\,}}

\newcommand{\G}{\Gamma}

\newtheorem{thm}{Theorem}[section]
\newtheorem{cor}{Corollary}[section]
\newtheorem{lem}{Lemma}[section]
\newtheorem{prop}{Proposition}[section]
\theoremstyle{definition}
\newtheorem{defn}{Definition}[section]
\theoremstyle{remark}

\newtheorem{rem}{Remark}[section]
\numberwithin{equation}{section}
\title{\bf Acoustic Scattering from Corners, Edges and Circular Cones}

\author{Johannes Elschner \thanks{ Weierstrass Institute, Mohrenstr. 39,
10117 Berlin,
Germany. Email: {\tt elschner@wias-berlin.de}},
Guanghui Hu \thanks{Weierstrass Institute, Mohrenstr. 39,
10117 Berlin,
Germany. Email: {\tt hu@wias-berlin.de}}}

\begin{document}
\maketitle

\begin{abstract}
Consider the time-harmonic acoustic scattering from a bounded  penetrable obstacle imbedded in an isotropic homogeneous medium. The obstacle is supposed to possess a circular conic point or an edge point on the boundary in three dimensions and a planar corner point in two dimensions.
The opening angles of cones and edges are allowed to be any number in $(0,2\pi)\backslash\{\pi\}$.
We prove that such an obstacle scatters any incoming wave non-trivially (i.e., the far field patterns cannot vanish identically), leading to the absence of real non-scattering wavenumbers. Local and global uniqueness results for the inverse problem of recovering the shape of a penetrable scatterers are also obtained using a single incoming wave. Our approach relies on the singularity analysis of the inhomogeneous Laplace equation in a cone.
\end{abstract}

\section{Introduction}\label{Sec:1}
Consider a time-harmonic acoustic wave  incident onto a bounded penetrable
scatterer $D\subset \R^n$  ($n=2,3$) embedded in a homogeneous isotropic medium.
The incident field $u^{in}$ is supposed to satisfy the Helmholtz equation
\be\label{eq:0}
\Delta w+k^2 w=0\quad\mbox{in}\quad \R^n,
\en
with the wavenumber $k>0$. Throughout the paper we suppose that $u^{in}$ does not vanish identically and that the
complement $D^e:=\R^n\backslash\overline{D}$ of $D$ is connected.
The acoustic properties of the scatterer can be described by the refractive index function $q\in L^\infty(\R^n)$ such that $q\equiv 1$ in $D^e$. Hence,
 the contrast function $1-q$ is supported in $D$.
 The wave propagation is then governed by the Helmholtz equation
\begin{equation}\label{eq:Helm}
\Delta u+k^2 q\, u=0\quad\mbox{in}\quad \R^n.
\end{equation}
In (\ref{eq:Helm}),
 $u=u^{in}+u^{sc}$ denotes the total wave where
$u^{sc}$ is the scattered field satisfying the Sommerfeld radiation condition
\begin{equation}\label{eq:radiation}
\lim_{|x|\rightarrow \infty} |x|^{\frac{n-1}{2}}\left\{ \frac{\partial u^{sc}}{\partial |x|}-ik u^{sc} \right\}=0.
\end{equation}
Across the interface $\partial D$, we assume the continuity of the total field and its normal derivative,
\be\label{TE}
u^+=u^-,\quad \partial_\nu u^+=\partial_\nu u^-\quad\mbox{on}\;\partial D.
\en
Here the superscripts $(\cdot)^\pm$ stand for the limits taken from outside and inside, respectively, and $\nu\in \s^{n-1}:=\{x\in\R^n: |x|=1\}$ is the unit normal on $\partial D$ pointing into $D^e$. The unique solvability of the scattering problem \eqref{eq:Helm}, \eqref{eq:radiation} and (\ref{TE}) in $H^2_{loc}(\R^n)$ is well known (see e.g., \cite[Chapter 8]{CK}). In particular, the Sommerfeld radiation condition (\ref{eq:radiation}) leads to  the asymptotic expansion
\begin{equation}\label{eq:farfield}
u^{sc}(x)=\frac{e^{ik |x|}}{|x|^{(n-1)/2}}\; u^\infty(\hat x)+\mathcal{O}\left(\frac{1}{|x|^{n/2}}\right),\quad |x|\rightarrow+\infty,
\end{equation}
 uniformly in all directions $\hat x:=x/|x|$, $x\in\mathbb{R}^n$. The function $u^\infty(\hat x)$ is an analytic function defined on $\s^{n-1}$ and is  referred to as the \emph{far-field pattern} or the \emph{scattering amplitude}.  The vector $\hat{x}\in\s^{n-1}$ is called the observation direction of the far field. The classical inverse medium scattering problem consists of the recovery of the refractive contrast $1-q$ or the boundary $\partial D$ of its support from the far-field patterns corresponding to one or several incident plane waves.
This paper is concerned with the following two questions:
\begin{description}
\item[(i)] Does a penetrable obstacle scatter any incident wave trivially (that is, $u^{sc}\equiv 0$) ?
\item[(ii)] Does the far-field pattern of a single plane wave uniquely determine the shape of a penetrable obstacle ?
\end{description}
A negative answer to the first question means that acoustic cloaking cannot be achieved using isotropic materials, while a positive answer to the second one implies uniqueness in inverse medium scattering with a single plane wave. It is widely believed that these assertions are true for a large class of scatterers; however, little progress has been made so far.
If $D$ trivially scatters any Herglotz wave function of the form
\ben
u^{\mathrm{in}}(x)=\int_{\s^{n-1}} \exp(ikx\cdot d)\,g(d)\,ds(d),\quad g\in L^2(\s^{n-1}),
\enn
then $\lambda=k^2$ is called  \emph{non-scattering energy}, or equivalently, $k$ is called \emph{non-scattering wavenumber}; see \cite{BLS}.
 A negative answer to the first question obviously leads to the absence of non-scattering energies. Moreover, it implies that the \emph{relative scattering operator} (or the so-called far-field operator \cite{CK}) has a trivial kernel and cokernel at every real wavenumber, which is required by a number of numerical methods in inverse scattering.
 Recall that $k>0$ is called an \emph{interior transmission eigenvalue} associated with the potential $q$ in $D$ if the coupling problem
 \be\label{ITP}\left\{\begin{array}{lll}
\Delta w + k^2 w=0,\quad \Delta u+k^2 q u=0&&\mbox{in}\quad D,\\
w=u,\qquad \qquad \;\partial_\nu w=\partial_\nu u &&\mbox{on}\quad \partial D.
\end{array}\right.
\en
 has at least one non-trivial solution $(w,u)\in H^1(D)\times H^1(D)$ such that $w-u\in H_0^2(D)$; see e.g., \cite{CGH2010,CP,CPJ,S}.
 A non-scattering wavenumber must be an interior transmission eigenvalue associated with the given potential, but not vice versa.
 An interior transmission eigenvalue $k$ is a non-scattering wavenumber only if the eigenfunction that satisfies the Helmholtz equation (\ref{eq:0}) in $D$ can be analytically extended as an incident wave into the whole space.
 We remark that the second question is more difficult than the first one. In fact, $D$ cannot scatter any incident wave trivially if $D$ could be uniquely determined by a single far-field pattern of any incoming wave. However, we do not know whether the reverse statement holds (see Theorem \ref{TH1} and Remark \ref{remark} (i)).

The answer to the uniqueness question provides an insight into whether or not the measurement data are sufficient to determine the unknowns, playing an important role in numerics (e.g., using optimization-based iterative schemes).
The shape identification problem in inverse scattering with a single far-field pattern is usually difficult and challenging, because it is a formally determined inverse problem, that is, the dimensions of the data and the  unknowns are the same.  For sound-soft obstacles, local uniqueness results were proved in \cite{CS1983,SU2004,Gintides}. Global uniqueness results have been obtained within the class of polyhedral or polygonal sound-soft and sound-hard scatterers (e.g., \cite{Rondi05,CY,EY06,HL14, LHZ06}), using the reflection principle for the Helmholtz equation under the Dirichlet and Neumann boundary conditions.
However, the proofs of these local and global uniqueness results do not apply to penetrable scatterers. See also \cite{Lax, Kirsch93} for the proof with infinitely many plane waves based on ideas of Schiffer and Isakov.
Earlier uniqueness results in inverse medium scattering were derived by sending plane waves with distinct directions at a fixed frequency (see e.g.,\cite{ElHu, Isakov90, Kirsch93}), which results in overdetermined inverse problems. Intensive efforts have also
been devoted to the unique determination of the variable contrast $1-q$  from  knowledge of the far-field patterns of all incident plane waves or by measuring the Dirichlet-to-Nuemann map of the Helmholtz equation. We refer to  \cite{Nac, SG} and \cite[Chapter 10.2]{CK} for the uniqueness in 3D and to recent results \cite{B,OYa} in 2D with certain regularity assumptions on the potential.

The study of non-scattering energies dates back to \cite{KS} in the case of a convex corner domain, with the main emphasis placed upon the exploration of the notion of \emph{scattering support} for an inhomogeneous medium. In the recent paper \cite{BLS}, it was shown that a penetrable scatterer having $C^\infty$-potentials with a rectangular corner scatters every incident wave non-trivially. The argument there is based on the use of complex geometric optics (CGO) solutions,
and the approach was later extended to the cases of a
 convex corner in $\R^2$ and a  circular conic corner in $\R^3$ whose opening angle is outside of a countable subset of $(0,\pi)$ (see \cite{PSV}).
  In the authors' previous work \cite{ElHu2015},  any corner in $\R^2$ and any edge in $\R^3$ are shown to be capable of scattering every incident wave non-trivially if the potential is real-analytic. In addition, the shape of a convex penetrable obstacle of  polygonal or polyhedral type can be uniquely determined by a single far-field pattern. The approach of \cite{ElHu2015} relies on the expansion of solutions to the Helmholtz equation with real-analytic potentials. The CGO-solution methods of \cite{PSV,BLS} also lead to uniqueness in shape identification but are confined so far to convex polygons in $\R^2$ and rectangular boxes in $\R^3$ with H\"older continuous potentials (see \cite{HSV}).

The aim of this paper is to verify uniqueness and the absence of real non-scattering wavenumbers in a more general setting. We shall consider
  \emph{curvilinear polygons} in $\R^2$, and \emph{curvilinear polyhedra} and \emph{circular cones} in $\R^3$ (see Section \ref{Results} for a precise definition) with an arbitrary piecewise H\"older continuous potential. We present a novel approach that relies heavily on the corner singularity analysis of solutions to the inhomogeneous Laplace equation in weighted H\"older spaces.
  If a penetrable obstacle scatters an incoming wave trivially or two distinct penetrable obstacles generate the same far-field pattern, one can always find a solution to the Helmholtz equation (\ref{eq:0}) in the exterior of an obstacle $D$ which extends analytically across a sub-boundary of $D$. However,
  we prove that in conic and wedge domains non-trivial solutions to the Helmholtz equation with certain boundary data cannot be analytically extended into a full neighborhood of the corner and edge points because of both the interface singularity and the medium discontinuity; see Lemmas \ref{lem}, \ref{lem:corner-2D}, \ref{lem:edge-3D}  and \ref{lem:curvilinear}.
  Our approach
 is different from those in \cite{PSV,ElHu2015} and extends the results of \cite{PSV,BLS,ElHu2015,HSV} to a large class of potential functions and corner domains. Moreover, we obtain a local uniqueness result for the inverse scattering problem with a single incoming wave and the global uniqueness within the class of convex polygons and polyhedra with flat surfaces; see Theorem \ref{TH2} and Corollary \ref{Corollary}. It should be remarked that our arguments are applicable to the case of more general incident fields (see Remark \ref{rem:point-source}), because only local properties of the Helmholtz equation are needed in our case of penetrable obstacles with singular boundary points. However, the far-field behaviour of the total field seems to be necessary in the unique determination of a general impenetrable scatterer.

The paper is organized as follows. Our results will be presented and verified in the subsequent Sections \ref{Results} and \ref{proof}. The proofs can be reduced to the analysis of a coupling problem between Helmholtz equations with different potentials near a boundary corner point; see Lemma \ref{lem}. We first carry out the proof of Lemma \ref{lem} for polygons in Section \ref{sec:3} and then generalize the arguments to polyhedra in Section \ref{3D-edges} by applying the partial Fourier transform. The techniques will be adapted to handle curvilinear polygons and  polyhedra, and circular cones in Sections \ref{sec: curvilinear} and \ref{sec:cone}.
In Sections \ref{Regularity} and \ref{Regularity-C}, we shall state the auxiliary solvability results for the Laplace equation in weighted Sobolev and H\"older spaces for  two and three dimensional cones, respectively. The proofs of several propositions that are used in Sections \ref{sec:2D}-\ref{sec:cone} will be carried out in the appendix.

 %Thus $u^{in}$ is allowed to a plane wave, a spherical point source wave emitting from some source position located in $\R^N\backslash\overline{D}$. %$d\in\mathbb{S}^{N-1}:=\{x\in \R^{N-1}:|x|=1\}$ is the incident direction.%, or the time-harmonic point source wave
%\be\label{source}
%u^{in}(x;y)=\left\{\begin{array}{lll}
%\frac{i}{4}H_0^{(1)}(k|x-y|),&& N=2,\\
%\frac{e^{ik|x-y|}}{4\pi|x-y|},&& N=3,
%\end{array}\qquad x\neq y,
%\right.
%\en
%where $y\in D^e:= \R^N\backslash\overline{D}$ denotes the position of the point source and $H_0^{(1)}$ the Hankel function of the first kind of order zero.

\section{Main results}\label{Results}
We introduce several notations before stating the main results.
For $j\in \N_0:=\{0\}\cup\N$,  $\nabla^j_x$ stands for the vector of all partial derivatives of order $j$ with respect to $x=(x_1,x_2,\cdots,x_n)\in \R^n$, i.e.,
\ben
\nabla^j_x u=\left\{\partial_{x_1}^{j_1}\partial^{j_2}_{x_2}\cdots \partial^{j_n}_{x_n}\,u(x):\quad j_1,j_2,\cdots,j_n\in\N_0,j_1+j_2+\cdots+j_n=j\right\}.
\enn In the particular case $j=1$, the notation $\nabla^1_x u=\nabla_x u$ means the gradient of $u$. If $j=0$, we have
$\nabla^0_x u=u$.
The spatial variable $x$ will be dropped when $\nabla^j$ is clearly understood from the context. Denote by $O$ the origin in $\R^n$.
Let $(r,\theta)$ be the polar coordinates of $x=(x_1,x_2)\in\R^2$.
 Define $\K=\K_\omega:=\{(r,\theta): r>0, 0<\theta<\omega\}$, a sector in $\R^2$ with the opening angle $\omega\in(0,2\pi)$ at the origin. Denote by $B_a(P):=\{x\in \R^n:|x-P|<a\}$ the ball centered at $P$ with radius $a>0$, and by $I$ the $n$-by-$n$ identity matrix in $\R^{n\times n}$. For simplicity we write $B_a(O)=B_a$.

We first introduce the concepts of  (planar) corner points in $\R^2$, and edge and circular conic points in $\R^3$;  see Figure \ref{Corner1} for illustration of planar corners of a curvilinear polygon.
%Recall that the number $l\in \N_0$ is the exponent for the regularity of the $l$-H\"older continuous potential $q$.
\begin{defn}\label{def-2d} (see e.g., \cite[Chapter 1.3.7]{MNP})
Let $D$ be a bounded open set of $\R^2$. The point $P\in\partial D$ is called corner point if
there exist a neighbourhood $V$ of $P$
, a diffeomorphism $\Psi$  of class $C^{\,2}$ and an angle $\omega=\omega(P)\in(0,2\pi)\backslash\{\pi\}$  such that
 \be\label{polygon}
 \nabla \Psi(P)=I\in \R^{2\times 2},\quad \Psi(P)=O,\quad
 \Psi(V\cap D)=\K_\omega\cap B_1.
 \en
 We shall say that $D$ is a curvilinear polygon, if for every $P\in \partial D$, (\ref{polygon}) holds with $\omega(P)\in(0,2\pi)$.
\end{defn}

\begin{defn}\label{def-3d}
Let $D\subset\R^3$ be a bounded open set. The point $P\in \partial D$ is called a vertex if
there exist a neighbourhood of $V$ of $P$,
 a diffeomorphism $\Psi$ of class $C^{\,2}$ and a polyhedral cone $\Pi$ with the vertex at $O$ such that
$\nabla \Psi(P)=I\in \R^{3\times 3}$, $\Psi(P)=O$ and $\Psi$ maps $V\cap \overline{D}$ onto a neighbourhood of $O$ in $\overline{\Pi}$. $P$ is called an edge point of $D$ if
\be\label{polyhedron}
\Psi(V\cap D)=(\K_\omega\cap B_1)\times(-1,1)
\en
for some $\omega(P)\in (0,2\pi)\backslash\{\pi\}$. We shall say that $D$ is a curvilinear polyhedron if, for every point $P\in\partial D$,  either (\ref{polyhedron}) applies  with $\omega(P)\in (0,2\pi)$ or $P\in\partial D$ is  a vertex.
\end{defn}

\begin{figure}
\begin{center}
\scalebox{0.35}{\includegraphics{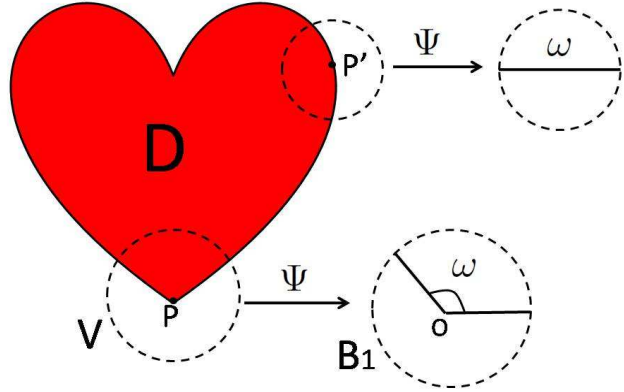}}
 \caption{$P\in \partial D$ is a corner of the curvilinear polygon $D$, whereas $P'$ is not a corner.}
\label{Corner1}
\end{center}
\end{figure}

A curvilinear polygon resp. polyhedron allows both curved and flat surfaces near a corner resp. edge point (see Figures \ref{Corner1} and \ref{cone-f}). The conditions (\ref{polygon}) and (\ref{polyhedron}) exclude  peaks at $O$ (for which the opening angle of the planar sector is $0$ or $2\pi$). %{\color{MyDarkRed} Suppose that $\partial D$ can be locally parameterized as $r=f(\theta)$ in two dimensions.  Then $(r, \theta_0)$ is a corner point of $\partial D$ if $f$ is piecewise $C^2$-smooth and $f'^+(\theta_0)\neq \pm f'^-(\theta_0)$.}

Let $(r,\theta,\varphi)$ be the spherical coordinates of $x=(x_1,x_2,x_3)\in\R^3$.
Let $\CC=\CC_\omega$ be an infinite circular cone in $\R^3$ defined as (see Figure \ref{cone-f})
\be\label{cone}
\CC:=\{(r,\theta,\varphi): r>0, 0<\theta<\omega, 0\leq\varphi<2\pi\}
\en
for some $\omega\in(0,\pi)\backslash\{\pi/2\}$.
Clearly, the vertex of $\CC$ is located at the origin and the opening angle of $\CC$ is $2\omega\in (0,2\pi)\backslash\{\pi\}$. The cone $\CC_\omega$ is identical with the half space $x_3>0$ if $\omega=\pi/2$.
\begin{defn}
We say that a bounded open set $D\subset\R^3$ has a circular conic point $P\in\partial D$ if
$ D\cap B_a(P)$ coincides with $\CC\cap B_a$ for some $a>0$ up to a coordinate translation or rotation.
$D$ is called a circular conical domain if it has at least one circular conic point.
%Analogously, $P$ is called plane corner resp. edge point of $D\subset\R^n$ ($n=2,3$) by the definition of \ref{def-2d} resp. \ref{def-3d}.
\end{defn}

\begin{figure}
\begin{center}
\scalebox{0.35}{\includegraphics{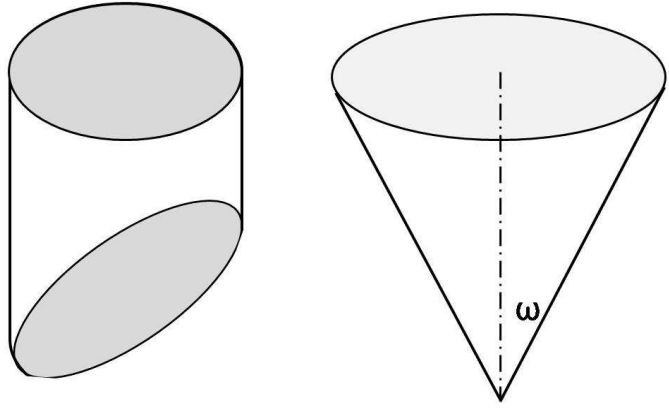}}
 \caption{Illustration of a curvilinear polyhedron (left) and a circular cone $\CC_\omega$  with the opening angle $2\omega\in (0,2\pi)\backslash\{\pi\}$ (right).}
\label{cone-f}
\end{center}
\end{figure}

 Let $D$ be a bounded penetrable obstacle in $\R^n$, with $O\in\partial D$ being a planar corner point in $\R^2$, and  an edge or circular conic point in $\R^3$. Denote by $W^{\kappa,p}$ and $H^\kappa=W^{\kappa,2}$ the standard Sobolev spaces. We make the following assumption on $q$ in a neighborhood of $O$.
\begin{description}
 \item[Assumption (a):] There exist $l\in \N_0$, $s\in(0,1), \epsilon>0$ such that
       \be\label{q}
       q\in C^{\,l,s}(\overline{D\cap B_\epsilon})\cap W^{l,\infty}(B_\epsilon),\quad
        \nabla^l\,(q-1)\neq 0 \quad \mbox{at}\quad O.
        \en
\end{description}
Note that the potential has been normalized to be one for $x\in D^e$ due to the homogeneity of the background medium, and that for $l\geq 1$
%the values of the potential in (\ref{q}) are obtained by taking the limits from $\overline{D}$.
 the relation $\nabla^l\,(q-1)\neq 0$ at $O$ means that at least one component of the vector $\nabla^l\,q(O)$ does not vanish.

By the assumption (a), $q$ is required to be $C^{\,l,s}$ continuous up to the boundary only in a neighborhood of $O$. The relation (\ref{q}) with $l=0$ means the discontinuity of $q$ at $O$, i.e., $q(O)\neq 1$, and  has been assumed in \cite{PSV,BLS,ElHu2015,HSV} in combination with other smoothness conditions on $q|_{\overline{D}}$ near $O$. A piecewise constant potential such that $q|_{\overline{D}}\equiv q_0\neq 1$ fulfills the assumption (a) with $l=0$.
When $l\geq 1$, it follows from the Sobolev imbedding relation $W^{l,\infty}(B_\epsilon)\subset C^{l-1}(B_\epsilon)$
that the function $q$ is $C^{l-1}$-smooth in $B_\epsilon$, implying that $q(x)=1+\mathcal{O}(|x|^l)$ as $|x|\rightarrow0$ in $D$.
Physically, this means a lower contrast of the material on $\overline{D\cap B_\epsilon}$ compared to the background medium.

The main results of this paper are stated as follows.

 \begin{thm}\label{TH1}
 Under the assumption (a), a penetrable obstacle with a planar corner point in $\R^2$, and with an edge or a circular conic point in $\R^3$
   scatters every incident wave non-trivially.
\end{thm}
Theorem \ref{TH1} implies the absence of real non-scattering wavenumbers in  curvilinear polygonal and polyhedral domains as well as in  circular conic domains. To answer the second question mentioned in Section \ref{Sec:1}, we present our uniqueness results in the following theorem and corollary (see Figure \ref{Corner2} for geometrical illustration).

 \begin{thm}\label{TH2} Let $D_j$ ($j=1,2$) be two penetrable obstacles in $\R^n$ ($n=2,3$). Suppose that the potentials $q_j$ associated to $D_j$ fulfill the assumption (a) for each corner,  edge and circular conic point. If $\partial D_2$ differs from $\partial D_1$ in the presence of a corner, edge or circular conic point lying on the boundary of the unbounded component of $\R^n\backslash\overline{(D_1\cup D_2)}$, then the far-field patterns corresponding to $D_j$ and $q_j$ incited by any incoming wave cannot coincide.
\end{thm}
\begin{figure}
\begin{center}
\scalebox{0.35}{\includegraphics{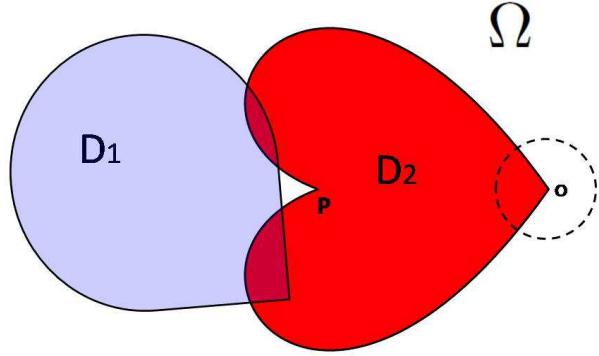}}
 \caption{ $D_1$ and $D_2$ cannot generate the same far-field pattern due to the presence of the corner point $O\in(\partial D_2\backslash\partial D_1)\cap \partial \Omega$, where $\Omega$ is the unbounded component of $\R^2\backslash\overline{(D_1\cup D_2)}$.
 The corner point $P$ lies on $\partial D_2\backslash\partial D_1$, but $P\notin \partial \Omega$.}
\label{Corner2}
\end{center}
\end{figure}
Clearly, the geometrical assumptions in Theorem \ref{TH2} are fulfilled if $D_1$ and $D_2$ are convex curvilinear polygons or polyhedra whose singular boundary points do not coincide. In particular, the latter always holds if $D_1$ and $D_2$ are two distinct convex polygons and polyhedra with piecewise flat boundaries.  Hence, we obtain the following global uniqueness results for the inverse scattering problem.
\begin{cor}\label{Corollary}
If the potential fulfills the assumption (a) near each corner resp. vertex, then
the shape of a convex penetrable polygon resp. polyhedron with flat sides can be uniquely determined by a single far-field pattern.
\end{cor}

\section{Proofs of Theorems \ref{TH1} and \ref{TH2}}\label{proof}
We first show the regularity of the total field in H\"older spaces depending on the smoothness of the potential.
\begin{prop}\label{P1}
Let $u\in H^2_{loc}(\R^n)$ be a solution to the Helmholtz equation $(\Delta+q)u=0$ in $\R^n$, $n=2,3$, and let
$\Omega\subset\R^n$ be a bounded Lipschitz domain. Assume $l\in \N_0$.
If $\nabla^{j}q\in L^\infty(\R^n)$ for all $j=0,1,\cdots,l$,  then $u\in C^{l+1,\alpha}(\overline{\Omega})\cap H^{l+2}(\Omega)$ for all $ \alpha\in[0,1)$.
%\begin{description}
%\item[(i)] If $q\in L^\infty(\R^n)$, then $u\in C^{1,\alpha}(\overline{\Omega})\cap H^{2}(\Omega)$ for all $ \alpha\in[0,1)$.
%\item[(ii)] If $q\in C^{\,l-1,s}(\R^n)$ and $\nabla^l q\in L^\infty(\R^n)$ for some $l\in \N$ and $s\in(0,1)$, then
%$u\in C^{l+1,\alpha}(\overline{\Omega})\cap H^{l+2}(\Omega)$ for all $ \alpha\in[0,1)$.
%\end{description}
\end{prop}
 \begin{proof}
 By Sobolev's imbedding theorem (see e.g.,\cite{GT}), we know that $u\in C(\R^n)$ for $n=2,3$. Therefore $qu\in L_{loc}^p(\R^n)$ for all $p\geq 2$, and by elliptic regularity $u\in W^{2,p}_{loc}(\R^n)$. Moreover, again applying Sobolev's imbedding theorem (see \cite[Theorem 7.26]{GT}) yields $W^{2,p}(\Omega)\subset C^{1,\alpha}(\overline{\Omega})$ for $\alpha=2-n/p-1$. This implies the assertion with $l=0$ by choosing the index $p\geq 2$ arbitrarily large. In the general case of $l\geq 1$, one can prove by induction that
  $qu\in W_{loc}^{l,p}(\R^n)$ for all $p\geq 2$, giving rise to  $u\in W^{l+2,p}_{loc}(\R^n)$ and  $u\in C^{l+1,\alpha}(\overline{\Omega})$ for all $\alpha\in[0,1)$.
\end{proof}
%Applying Proposition \ref{P1} to the total field of our scattering problem, we see that $u\in C^{l+1,\alpha}(\overline{B_\epsilon\cap D})\cap H^{l+2}(B_\epsilon\cap D)$  for some $\epsilon>0$ and for all $ \alpha\in[0,1)$,

The proofs of our results essentially rely on the following lemma.
\begin{lem}\label{lem} Let $D\subset \R^n$ ($n=2,3$) be a bounded domain. Assume that
  $q\in L^\infty(D)$ satisfies the assumption (a) near the boundary point $O\in\partial D$ and that $q\equiv 1$ in $\R^n\backslash\overline{D}$.
 It is supposed that one of the following cases holds:
  \begin{description}
 \item[(i)] $O$ is a planar corner point if $D\subset \R^2$ is a curvilinear polygon;
 \item[(ii)] $O$ is an edge point if $D\subset \R^3$ is a curvilinear polyhedron;
 \item[(iii)] $O$ is  the vertex of some circular cone if $D\subset \R^3$ is a circular conic domain.
 \end{description}
 For $\epsilon>0$ sufficiently small, let $\Gamma_\epsilon= \partial D\cap B_\epsilon$
  be a sub-boundary of $\partial D$ containing $O$.
  If the solution pair $u_j\in H^2(B_\epsilon)$ ($j=1,2$) solves the coupling problem
\be\label{eq:30}\begin{split}
&\Delta u_1+ k^2 u_1=0,\qquad \Delta u_2+ k^2 q\, u_2=0 \quad \quad\mbox{in}\quad B_\epsilon,\\
&\partial_\nu^{j}(u_1-u_2)=0\quad \mbox{on}\quad \Gamma_{\epsilon},\; \qquad\qquad j=0,1,2,\cdots,l+1,
\end{split}
\en
then $u_1=u_2\equiv 0$ in $B_\epsilon$. Here the number $l\in \N_0$ is specified by the regularity of $q$ in the assumption (a).
\end{lem}
Note that when $l=0$, the transmission conditions in (\ref{eq:30}) are reduced to the classical TE transmission conditions:
\ben
u_1=u_2,\quad \partial_\nu u_1=\partial_\nu u_2\qquad \mbox{on}\quad \Gamma_{\epsilon}.
\enn
Lemma \ref{lem} with $l=0$ can be interpreted as follows: The Cauchy data of non-trivial solutions to the two Helmholtz equations in (\ref{eq:30}) do not coincide on the boundary $\G_{\epsilon}$  if the values of the potentials involved are not identical at $O\in \G_{\epsilon}$. In other words, there are non-trivial solutions to the Helmholtz equation
  $\Delta u_1+ k^2 u_1=0$
  in $ D^e\cap B_\epsilon$ that cannot be analytically extended into a full neighborhood of $O$ due to both the interface singularity at $O\in \Gamma_{\epsilon}$ and the discontinuity of $q$ at $O$. For $l\geq 1$, the transmission conditions in (\ref{eq:30}) are well defined by Proposition  \ref{P1}.
Below we shall prove our results by applying Lemma \ref{lem}.

{\bf Proof of Theorem \ref{TH1}.} Consider the scattering problem (\ref{eq:0})-(\ref{eq:farfield}) for the penetrable obstacle $D\subset \R^n$. Denote by $O\in\partial D$ the planar corner point in $\R^2$,  the edge point or the circular conic point in $\R^3$.
By Proposition \ref{P1}, the total field $u$ has the regularity
\ben
u\in C^{l+1,\alpha}(\overline{D\cap B_\epsilon})\cap H^{l+2}( D\cap B_\epsilon)\quad\mbox{for all}\quad \alpha\in[0,1)
\enn
under the assumption (a). Hence, if the scattered field vanishes identically, there hold the transmission conditions
\ben
\partial_{\nu}^{j}\, u=\partial_{\nu}^{j}\, u^{in}\quad\mbox{on}\quad \Gamma_{\epsilon},\quad j=0,1,\cdots,l+1,
\enn
where $\Gamma_{\epsilon}\subset\partial D$ contains $O$.
Now, applying Lemma \ref{lem} to $u_1=u^{in}$ and $u_2=u$ gives $u^{in}\equiv0$ in $B_\epsilon$.
By unique continuation, $u^{in}\equiv 0$ in $\R^n$, which is a contradiction.
\hfill$\Box$

{\bf Proof of Theorem \ref{TH2}.} Denote by $(D_j, q_j)$ ($j=1,2$) the two penetrable obstacles and the associated potentials.
If the far-field patterns incited by some incoming wave
corresponding to $(D_1,q_1)$ and $(D_2,q_2)$ coincide, then by Rellich's lemma the scattered fields
must also coincide in the unbounded component $\Omega$ of $\R^n\backslash\overline{(D_1\cup D_2)}$.
Suppose without loss of generality that there exists a corner $O\in \partial D_2\cap\partial \Omega$
such that $O\notin \partial D_1$ (see Figure \ref{Corner2}). Then, one can find a small $\epsilon>0$ such that $ D_1\cap B_\epsilon=\emptyset$.
Applying Lemma \ref{lem} to the domain $D:= D_2\cap B_\epsilon$ with
$u_j$ being the total fields corresponding to $(D_j, q_j)$, $j=1,2,$ we finally get $u_1\equiv 0$ in $D$ and thus
$u_1\equiv 0$ in $\R^n$. This implies that the scattered field $u_1^{sc}:=u_1-u^{in}$
can be extended to the whole space as a solution to the Helmholtz equation
with the wavenumber $k^2$. Hence, $u_1^{sc}\equiv 0$ and thus $u^{in}\equiv 0$ in $\R^n$.
This contradiction implies that $(D_1,q_1)$ and $(D_2,q_2)$ cannot generate identical far-field patterns.
\hfill$\Box$

\begin{rem}\label{rem:point-source}
The proofs of our results carry over to all non-vanishing incident fields that satisfy the Helmholtz equation (\ref{eq:0}) in a neighborhood of $D$, including the incident point source waves of the form
\ben
u^{in}(x;y)=\left\{\begin{array}{lll}
\frac{i}{4}H_0^{(1)}(k|x-y|),&& n=2,\\
\frac{e^{ik|x-y|}}{4\pi|x-y|},&& n=3,
\end{array}\qquad x\neq y,\quad y\in D^e.
\right.
\enn Here $H_0^{(1)}$ denotes the Hankel function of the first kind of order zero.
\end{rem}
\begin{rem}\label{remark}
It is not straightforward to generalize the global uniqueness result of Corollary \ref{Corollary} to the class of all curvilinear  polygons and polyhedra, because in general one cannot always find a singular boundary point in a neighbourhood of which
the wave field is analytic; see the proof of Theorem \ref{TH2}.
 Due to the same reason, our approach for proving Corollary \ref{Corollary} does not apply to non-convex polygons and polyhedra. For a non-convex scatterer, the unique determination of its convex hull follows from the proof of Theorem \ref{TH2}.
     We refer to \cite{Rondi05,CY,EY06,HL14,LHZ06} where non-convex impenetrable polygons and polyhedra were treated, relying on reflection principles for the Helmholtz equation in combination with properties of incident  plane or point source waves.
\end{rem}
\begin{rem}
Lemma \ref{lem} does not hold in the absence of interface singularities on $\Gamma_{\epsilon}$, for instance, if $\Gamma_{\epsilon}$ is an analytic surface. To see this, we let $l=0$, $q|_{D}\equiv q_0\neq 1$, and suppose that $\Gamma_{\epsilon}=\{-\epsilon<x_1<\epsilon\}\subset\R^2$ is a line segment. Then it is easy to check that
\ben
u_1=e^{-ik x_2}+\frac{1-q_0}{1+q_0}e^{ik x_2},\quad u_2=\frac{2}{1+q_0}e^{-ikq_0x_2},
\enn are non-trivial solutions to (\ref{eq:30}). In fact, $u_1$ and $u_2$ denote respectively the unique total and transmitted fields in the upper and lower half spaces incited by the incoming wave $\exp(-ik x_2)$ incident onto $x_2=0$ from above.
\end{rem}

The rest of this paper is devoted to the proof of Lemma \ref{lem} for curvilinear polygons and polyhedra in Sections \ref{sec:2D}-\ref{sec: curvilinear}, and for circular cones in Section \ref{sec:cone}.
In the case of $l=0$ and a real-analytic refractive index $q$ on $\overline{D\cap B_\epsilon}$, an alternative and more straightforward proof was presented in \cite{ElHu2015} for polygons and polyhedra with flat surfaces.

\section{Corners in 2D always scatter}\label{sec:2D}
This section is concerned with the acoustic scattering from a penetrable polygon with a piecewise linear boundary in $\R^2$.
The curvilinear polygons will be treated later in Section \ref{sec: curvilinear}. Our approach relies on the singularity analysis of the inhomogeneous Laplace equation in a sector.
We refer to the fundamental paper \cite{Kondratiev} and the monographs \cite{Grisvard,MNP,NP} for a general regularity theory of elliptic boundary value problems in domains with non-smooth boundaries.

\subsection{Solvability of the Laplace equation in a sector}\label{Regularity}

We introduce two classes of weighted spaces on the sector $\K$ introduced in Section \ref{Results}. For $\kappa\in \N_0$ and $\beta\in \R$,
the weighted Sobolev spaces $V_\beta^{\kappa}(\K)$ are defined as the completion of $C_0^\infty(\K)$ with respect to the norm
\ben
||u||_{V_\beta^{\kappa}(\K)}=\left\{\sum_{j\in \N_0, j\leq \kappa}\int_{\K} r^{2(\beta-\kappa+j)}\,|\nabla^j_x\,u(x)|^2\,dx \right\}^{1/2}.
\enn Denote by $\Lambda_\beta^{\kappa,\alpha}(\K)$ the weighted H\"older spaces endowed with the norm
\ben
||u||_{\Lambda_\beta^{\kappa,\alpha}(\K)}&=&\sup_{x,y\in \K } |x-y|^{-\alpha}\,\big| |x|^\beta \nabla^{\kappa}_x\,u(x)- |y|^\beta \nabla^{\kappa}_y\,u(y)\big| \\
&&+\sup_{x\in \K } \sum_{j\in\N_0,j\leq \kappa} |x|^{\beta-\alpha-\kappa+j}\,|\nabla^j_x u(x)|
\enn
for $\alpha\in(0,1)$. If $u\in \Lambda_\beta^{\kappa,\alpha}(\K)$, then $\nabla^j u\in \Lambda_\beta^{\kappa-j,\,\alpha}(\K)$ for all $j=0,1,\cdots, \kappa$. In addition,
the inclusion
 $\Lambda_\beta^{\kappa,\alpha}(\K)\subset \Lambda_{\beta+1}^{\kappa,\alpha}(\K)$ holds for functions with a compact support in $\K$.

Let $\Delta_D$ resp. $\Delta_N$ be the operator of the Dirichlet resp. Neumann problem corresponding to the inhomogeneous Laplace equation with the homogeneous boundary condition on $\partial \K$. In this subsection the operators $\Delta_D$ and $\Delta_N$ will act on the spaces
\ben
\Lambda_{\beta, D}^{\kappa,\alpha}(\K)&:=&\left\{u\in  \Lambda_{\beta}^{\kappa,\alpha}(\K): u=0\;\mbox{on}\;\partial \K  \right\},\\
V_{\beta, D}^{\kappa}(\K)&:=&\left\{u\in  V_{\beta}^{\kappa}(\K): u=0\;\mbox{on}\;\partial \K  \right\},
\enn and
\ben
\Lambda_{\beta, N}^{\kappa,\alpha}(\K)&:=&\left\{u\in  \Lambda_{\beta}^{\kappa,\alpha}(\K): \partial_\nu u=0\;\mbox{on}\;\partial \K  \right\},\\
V_{\beta, N}^{\kappa}(\K)&:=&\left\{u\in  V_{\beta}^{\kappa}(\K): \partial_\nu u=0\;\mbox{on}\;\partial \K  \right\}
\enn respectively.
%The following propositions are special cases of the solvability results in.
In the following we state solvability results for the Laplace equation in the weighted spaces $V^2_{\beta}(\K)$ and $\Lambda_\beta^{2,\alpha}(\K)$. %which are the special case of $\kappa=0$ in \cite{NP}.

\begin{prop}\label{P2}(\cite[Chapter 2, Proposition 2.5]{NP})
\begin{description}
\item[(i)] The operator $\Delta_D:  V_{\beta, D}^{2}(\K)\rightarrow V_{\beta}^{0}(\K)$ is an isomorphism if $ 1-\beta\neq j\pi/\omega$ for all $j\in \Z\backslash\{0\}$.
\item[(ii)] The operator $\Delta_N:  V_{\beta, N}^{ 2}(\K)\rightarrow V_{\beta}^{0}(\K)$ is an isomorphism if $ 1-\beta\neq j\pi/\omega$ for all $j\in \Z$.
\end{description}
\end{prop}
\begin{prop}\label{P3}(\cite[Chapter 3, Theorem 6.11]{NP})
\begin{description}
\item[(i)] The operator $\Delta_D:  \Lambda_{\beta, D}^{ 2,\alpha}(\K)\rightarrow \Lambda_{\beta}^{0,\alpha}(\K)$ is an isomorphism if $ 2+\alpha-\beta\neq j\pi/\omega$ for all $j\in \Z\backslash\{0\}$.
\item[(ii)] The operator $\Delta_N:  \Lambda_{\beta, N}^{ 2,\alpha}(\K)\rightarrow \Lambda_{\beta}^{0,\alpha}(\K)$ is an isomorphism if $ 2+\alpha-\beta\neq j\pi/\omega$ for all $j\in \Z$.
\end{description}
\end{prop}

\begin{prop}\label{P4}(\cite[Chapter 2, Proposition 2.12]{NP})
Let $\gamma_1<\gamma\leq  2$  and assume $ 2+\alpha-\beta\neq j\pi/\omega$ for $\beta=\gamma,\gamma_1$ and for all $j\in \N$. Moreover, let $f\in \Lambda_{\gamma}^{0,\alpha}(\K)\bigcap \Lambda_{\gamma_1}^{0,\alpha}(\K)$ and denote by $v_\beta$ the unique solution of the Dirichlet problem $\Delta_D v=f\in \Lambda_\beta^{0,\alpha}(\K)$ in $\Lambda_{\beta,D}^{2,\alpha}(\K)$. Then we have the relation
\be\label{Dirichlet}
v_{\gamma_1}=v_{\gamma}+\sum_j C_j\,r^{j\pi/\omega}\,\sin[(j\pi/\omega)\theta],\quad C_j\in \C,
\en
where the sum is taken over all $j\in\N$ such that $j\pi/\omega\in( 2+\alpha-\gamma, 2+\alpha-\gamma_1)$. For the Neumann problem, the $\sin$ functions in (\ref{Dirichlet}) should be replaced by the $\cos$ functions.
\end{prop}

Let $\mathbb{P}_{\kappa}$ be the set of homogeneous polynomials of degree $\kappa\in \N_0$ in $\R^n$. Below we present a special solution to the two-dimensional Laplace equation when the right hand side is a homogeneous polynomial; see \cite[Section 2.3.4]{NP}.
\begin{prop}\label{P5-2d}
Consider the inhomogeneous Dirichlet problem $\Delta_D v=p_{\kappa}\in \mathbb{P}_{\kappa}$ in $\K\in\R^2$. There exists a special solution of the form
\be\label{special-solution}\begin{split}
&v=q_{\kappa+ 2}&&\mbox{if}\quad (\kappa+ 2)\omega/\pi\notin \N,\\
&v=q_{\kappa+ 2}+C_D\,r^{\kappa+ 2}\left\{ \ln r \sin (\kappa+ 2)\theta+\theta\cos(\kappa+ 2)\theta    \right\}&&\mbox{if}\quad  (\kappa+2)\omega/\pi\in \N
\end{split}
\en for some $C_D\in \C$ and $q_{\kappa+ 2}\in \PP_{ \kappa+2}$ satisfying $\Delta q_{\kappa+2}=p_{\kappa}$.

For the Neumann problem $\Delta_N v=p_{\kappa}\in \mathbb{P}_{\kappa}$, a special solution takes the same form as (\ref{special-solution}) when $ (\kappa+2)\omega/\pi\notin \N$,  but with
\ben
v=q_{\kappa+ 2}+C_N\,r^{\kappa+ 2}\left\{ \ln r \cos ( \kappa+2)\theta-\theta\sin( \kappa+2)\theta \right\},\quad C_N\in\C,
\enn
if $ (\kappa+2)\omega/\pi\in \N$.
\end{prop}

\subsection{Proof of Lemma \ref{lem} for polygons}\label{sec:3}

Let $\K\subset \R^2$ be an infinite sector with the angle $\omega\in(0,2\pi)\backslash\{\pi\}$. Recall that $B_1$ is the unit disk centered at the origin $O$. Assume $q\in C^{\,l,s}(\overline{\K\cap B_1})$ for some $l\in \N_0$, $s\in(0,1)$ satisfying $q\equiv1$ in $B_1\backslash\overline{\K}$.
 Consider the coupling problem between the Helmholtz equations
\be\begin{split}\label{eq:1.0}
&\Delta u_1+ k^2 u_1=0,\qquad \Delta u_2+ k^2 q u_2=0 \quad \mbox{in}\quad B_1,\\
&\partial_\nu^{j}(u_1-u_2)=0\quad \mbox{on}\quad \partial \K\cap B_1,\; j=1,2,\cdots,l+1,
\end{split}
\en
where $\partial_\nu^{j}$ denotes the normal derivative of order $j$ at $\partial \K$ and $\nu$ is the unit normal pointing into the exterior of $\K$.
The proof of Lemma \ref{lem} for a polygon with piecewise linear boundary follows straightforwardly from the lemma below, which implies that corners in 2D always scatter.
\begin{lem}\label{lem:corner-2D}
Let $u_1,u_2\in H^2(B_1)$ be solutions to (\ref{eq:1.0}), and suppose
 that $q$ satisfies the assumption (a) near the corner $O$ with $D:=\K\cap B_1$.
Then $u_1=u_2\equiv0$ in $B_1$.
\end{lem}
 Lemma \ref{lem:corner-2D} will be proved by applying the solvability results of the Laplace equation in the weighted spaces introduced in Section \ref{Regularity}.
  For simplicity we write $\Lambda_{\beta}^{\,\kappa,\alpha}=\Lambda_{\beta}^{\,\kappa,\alpha}(\K)$ and $V_\beta^\kappa=V_\beta^\kappa(\K)$ to drop the dependence on the sector $\K$ in this subsection.

\begin{proof} Obviously, $u_1$ is real-analytic in $B_1$ and
by Proposition \ref{P1},
$$u_2\in C^{\,l+1,\alpha}(\overline{B}_1)\cap H^{\,l+2}(B_1)\quad\mbox{for all}\quad\alpha\in[0,1).$$ Hence, the traces of $u_1$ and $u_2$ on $\partial \K\cap B_1$ occurring in (\ref{eq:1.0}) are all well defined.
For clarity we shall divide the proof into five steps.

\textbf{Step 1.} Setting $u:=u_1-u_2$, we have
\ben
&&\Delta u+k^2qu=k^2(1-q)u_1 \quad\mbox{in}\quad \K\cap B_1,\\
&&\partial_\nu^{j} u=0 \qquad\mbox{on}\quad \partial \K\cap B_1,\;j=1,2,\cdots, l+1.
\enn
Let $\tilde{v}=\nabla^l u$. Then $\tilde{v}\in C^{1,\alpha}(\overline{\K\cap B_1})\cap H^{2}(\K\cap B_1)$ solves the following Cauchy problem for the Laplace equation with an inhomogeneous right hand side
\ben
\Delta\tilde{v}=-k^2\nabla^l(qu)+k^2\nabla^l(hu_1)\quad\mbox{in}\quad \K\cap B_1,
\quad \tilde{v}=\partial_\nu \tilde{v}=0 \quad\mbox{on}\quad \partial \K\cap B_1,
\enn where $h:=1-q$. Here and in the following a scalar differential operator is assumed to act componentwise on a vector function.

 We shall analyze the singularity of $\tilde{v}$ near the corner $O$.
Since the solvability results in Propositions \ref{P2}-\ref{P4} refer to the case of an infinite cone, we will introduce a new boundary value problem defined over $\K$. For this purpose, we
choose a cut-off function $\chi\in C_0^\infty(\overline{\K})$ such that
$\chi\equiv 1$ in $\K\cap B_{1/2}$ and $\chi\equiv 0$ in $ \K\cap B^e_{1}$. Define a new function $v$ as
\ben
v:=\left\{\begin{array}{lll}
\chi\, \tilde{v} && \mbox{in}\quad \K\cap \overline{B}_1,\\
0&&\mbox{in}\quad \K\cap B_1^e.
\end{array}\right.
\enn
Introduce the commutator in $\K\cap B_1$:
$$[\Delta, \chi]\tilde{v}:=\Delta(\chi\tilde{v})-\chi\Delta\tilde{v}=\tilde{v}\Delta\chi+2\nabla \tilde{v}\cdot\nabla\chi$$
and extend $[\Delta, \chi]\tilde{v}$, $q, h, u$ and $u_1$ by zero to $\K\cap B_1^e$.  Simple calculations show that
\be\label{eq:2}\begin{split}
&\Delta v=-k^2\chi\,\nabla^l(qu)+k^2\chi\,\nabla^l(hu_1)-[\Delta,\chi]\tilde{v}=:f\quad\mbox{in}\quad \K,\\
&v=\partial_\nu v=0 \quad\mbox{on}\quad \partial \K.
\end{split}
\en
We shall study the boundary value problem (\ref{eq:2}) in the weighted H\"older spaces $\Lambda_\beta^{2,\alpha}(\K)$ ($\beta\leq 1$) introduced in Section \ref{Regularity} where the weight $\beta$ will be improved step by step. %Choose the H\"older exponent $0<\alpha<s$ where $s$ is the exponent for $q$.
The inhomogeneous term $f$ in (\ref{eq:2}) belongs to $C^{\,0,\alpha}(\overline{\K})$ and thus to $\Lambda_{1}^{\,0,\alpha}$ for all $\alpha\leq s$, while $v\in C^{1,\alpha}(\overline{\K})\cap H^{2}(\K)$.
Recall that $s$ is the H\"older exponent of $q$.

\textbf{Step 2.} We show that
 $v\in \Lambda_{1,D}^{2,\alpha}\cap \Lambda_{1,N}^{2,\alpha}$ if the H\"older exponent $0<\alpha<s$ is sufficiently small.

 First it holds that $v\in V_0^2$, since $v$ has compact support, $v\in H^2(\K)$ and by the vanishing Cauchy data,
 \ben
 r^{-2} |v|+ r^{-1} |\nabla v|=\mathcal{O}(r^{\alpha-1})\quad\mbox{as}\quad r\rightarrow 0.
 \enn
 Hence, by Proposition \ref{P2} with $\beta=0$, $v$ is the unique solution of (\ref{eq:2}) in the weighted Sobolev space $V_{0,D}^{2}\cap V_{0,N}^{2}$; note that $1\neq j\pi/\omega$ for all $j\in \N_0$ since the opening angle $\omega\in(0,2\pi)\backslash\{\pi\}$.
 On the other hand, since $f\in \Lambda_{\beta}^{0,\alpha}$ for all $\beta\geq1$,
 by Propositions \ref{P3} and \ref{P4} there are unique solutions $v_{D/N}$
 of the first equation in (\ref{eq:2})
 satisfying $v_D\in \Lambda_{\beta,D}^{2,\alpha}$ and
$v_{N}\in \Lambda_{\beta,N}^{2,\alpha}$
for all $\beta\geq 1$ sufficiently close to $1$ and $\alpha>0$ sufficiently small.
Note that, for those $\alpha$ and $\beta$, $2+\alpha-\beta\neq j\pi/\omega$ for all $j\in \N$.
   %Choose $\alpha>0$ sufficiently small avoiding all critical values in Propositions \ref{P2} and \ref{P3}
    %and $\gamma>1+\alpha$ sufficiently close to 1 such that
 %\ben
 %2+\alpha\neq j\pi/\omega, \quad 2+\gamma+\alpha\neq j\pi/\omega,\quad j\pi/\omega\notin (2+\alpha-\gamma,1+\alpha)
 %\enn for all $j\in \Z$.
% Then it follows from Proposition \ref{P4} (with $\gamma_1=1$) that
%$v_{\gamma,D/N}=v_{1,D/N}=:v_{D/N}$.
%Choosing $\beta=1$ in Proposition \ref{P2}, we get
Moreover,  $v_{D/N}\in\Lambda_{1}^{2,\alpha}$ implies that $\chi v_{D/N}\in V_0^2$.
Since also
$v_{D/N}\in\Lambda_{\beta}^{2,\alpha}$ for some $\beta>1$, it is easy to check that $(1-\chi)v_{D/N}\in V_0^2$.
Therefore, we obtain $v_{D/N}\in V_0^2$,  implying that $v=v_D=v_N$ and the required regularity of $v$ in this step.
%Next, we prove that $\tilde{u}\in \Lambda_{1,D}^{\,l+2,\alpha}\cap \Lambda_{1,N}^{\,l+2,\alpha}$, improving the weight $\beta=l+1$ of the smoothness of $v$ around $O$ to $\beta=1$.
 %This follows immediately from the previous paragraph if $l=0$. Now we assume that $l\geq 1$. Observing that
% \ben
 %\tilde{u}\in \Lambda_{l+1}^{\,l+2,\alpha}\subset \Lambda_{l}^{\,l,\alpha},\qquad
 %k^2\chi\, h\, u_1-[\Delta,\chi]u\in \Lambda_{1}^{\,l,\alpha}\subset \Lambda_{l}^{\,l,\alpha},
% \enn
% the function $\tilde{f}$ on the right hand side of (\ref{eq:2.0}) belongs to $\Lambda_{l}^{\,l,\alpha}$. Hence,
% by Proposition \ref{P3} with $\beta=l$, there are unique solutions $\tilde{u}_D\in \Lambda_{l,D}^{\,l+2,\alpha}$ and
%$\tilde{u}_N\in \Lambda_{l,N}^{\,l+2,\alpha}$ of (\ref{eq:2.0}) with $j=0$ and $j=1$, respectively. Since $\Lambda_{l,D/N}^{\,l+2,\alpha}\subset \Lambda_{l+1,D/N}^{\,l+2,\alpha}$, we obtain $\tilde{u}=\tilde{u}_D=\tilde{u}_N\in \Lambda_{l,D}^{\,l+2,\alpha}\cap \Lambda_{l,N}^{\,l+2,\alpha}$, which has improved the weight $\beta=l+1$ to $\beta=l$. The desired regularity follows by repeating this process and taking into account the fact that $q\in \Lambda_{1}^{\,l,\alpha}$.
%
%To prepare the subsequent steps, we define $v:=\nabla^{l}\tilde{u}$ and $f:=\nabla^{l} \tilde{f}$. Then $v\in \Lambda_{1,D}^{\,2,\alpha}\cap \Lambda_{1,N}^{\,2,\alpha}$ solves the equation
%\be\label{eq:2}
%\Delta v=f\quad\mbox{in}\quad K,\qquad f=-k^2\nabla^{l}(qu)+k^2\chi \nabla^{l}(u_1h)+F\in \Lambda_{1}^{0,\alpha},
%\en where $F\equiv0$ in $\K\cap B_{1/2}$.

\textbf{Step 3.} We show that  $ f\in \Lambda_{0}^{\,0,\alpha},  v\in \Lambda_{0}^{2,\alpha}$ for $\alpha>0$ sufficiently small, and $u_1(O)=0$.

From the regularity assumption on $q$ it follows that
 \be\label{eq:27} \nabla^{j}h(O)= 0, \;j=0,1,\cdots,l-1,\quad \nabla^{l}h(O)\neq 0.\en
 The last relation means that $\partial_{x_1}^{l_1}\partial_{x_2}^{l_1}h(O)\neq 0$ for some $l_1,l_2\in\N_0$ such that $l_1+l_2=l$.
Using (\ref{eq:27}) and the fact that $v\in \Lambda_{1}^{2,\alpha}$ we get
\ben
\chi\nabla^{l}(qu)\in \Lambda_{1}^{2,\alpha}\subset \Lambda_{0}^{0,\alpha},\quad k^2\chi\,[\nabla^{l}(h\,u_1)- \nabla^{l}h(O)\;u_1(O) ]\,\in \Lambda_{0}^{0,\alpha}.
\enn
 Hence, the right hand side of (\ref{eq:2}) takes the form
\be\label{eq:5}
f=\chi p_0+f_0,\qquad\quad  p_0:=k^2\,\nabla^{l}h(O) \,u_1(O),\quad f_0:=f-\chi p_0\in \Lambda_{0}^{0,\alpha},
\en
 that is, $\chi p_0$ is the only part of $f\in \Lambda_{1}^{0,\alpha}$ that does not belong to $\Lambda_0^{0,\alpha}$. Therefore, it suffices to verify the vanishing of the constant vector $p_0$ in this step.

Consider the boundary value problems
 \be\label{w}
 \Delta_D v_0=p_0,\quad \Delta_N v_0=p_0\quad\mbox{on}\quad\overline{\K}.
 \en Applying Proposition \ref{P5-2d} with $\kappa=0$ yields special solutions $v_{0,D}$, $v_{0,N}$ to (\ref{w}) of the form
 \be\label{eq:3}\begin{split}
 v_{0,D}&=&q_{ 2,D}+c_D\,r^{ 2}\left\{ \ln r \sin 2\theta+\theta\cos2\theta \right\},\\
  v_{0,N}&=&q_{ 2,N}+c_N\,r^{ 2}\left\{ \ln r \cos 2\theta-\theta\sin2\theta \right\},
  \end{split}
 \en where $q_{ 2,D/N}\in \PP_{ 2}$, $c_{D/N}\in\C$ satisfy
 \ben
 \Delta q_{ 2,D/N}=p_0,\qquad
 c_{D/N}=0\quad\mbox{if}\quad 2\omega/\pi\notin \N.
 \enn
 For the (unique) solution $v\in \Lambda_{1,D}^{2,\alpha}\cap \Lambda_{1,N}^{2,\alpha}$
 of the problem (\ref{eq:2}), we
 set $$w_{0,D/N}:=v-\chi\, v_{0,D/N}\in \Lambda_1^{ 2,\alpha}.$$
 Using (\ref{eq:5}),
 one can readily check that
 \ben
 \Delta w_{0,D}=f_0-[\Delta,\chi]\,v_{0,D}=:g_{0,D}\in \Lambda_0^{0,\alpha}\cap \Lambda_1^{0,\alpha},\\
 \Delta w_{0,N}=f_0-[\Delta,\chi]\,v_{0,N}=:g_{0,N}\in \Lambda_0^{0,\alpha}\cap \Lambda_1^{0,\alpha}.
 \enn
We apply Proposition \ref{P4} with $\gamma_1=0$ and $\gamma=1$ to the previous two boundary value problems to get the unique solutions in $\Lambda_1^{ 2,\alpha}$  of the form
\be\label{eq:4}\begin{split}
 w_{0,D}= \chi\,\sum_j d_{D,j}\, r^{ j\pi/\omega} \,\sin[(j\pi/\omega)\theta]  +\tilde{w}_{D},\qquad d_{D,j}\in\C,\quad \tilde{w}_{D}\in \Lambda_{0,D}^{ 2,\alpha},\\
  w_{0,N}=\chi\, \sum_j d_{N,j}\, r^{ j\pi/\omega} \,\cos [(j\pi/\omega)\theta] +\tilde{w}_{N},
  \qquad d_{N,j}\in\C,\quad \tilde{w}_{N}\in \Lambda_{0,N}^{ 2,\alpha},
  \end{split}
\en
where the sums in (\ref{eq:4}) are both taken over all $j\in \N$ such that
 $j\pi/\omega\in ( 1+\alpha,  2+\alpha)$, or equivalently, $j\pi/\omega\in ( 1+\alpha,  2]$.
Comparing (\ref{eq:3}), (\ref{eq:4}) and recalling that $v$ solves both the Dirichlet and Neumann boundary value problems, we obtain the following expressions as $r\rightarrow 0$:
\be\label{eq:21}\begin{split}
v&=\sum_j d_{D,j}\, r^{ j\pi/\omega} \,\sin[(j\pi/\omega)\theta]+q_{ 2,D}+c_D\,r^{ 2}\left\{ \ln r \sin 2\theta+\theta\cos2\theta \right\}+\mathcal{O}(r^{2+\alpha})\\
&= \sum_j d_{N,j}\, r^{ j\pi/\omega} \,\cos [(j\pi/\omega)\theta]
  +q_{ 2,N}+c_N\,r^{ 2}\left\{ \ln r \cos 2\theta-\theta\sin2\theta \right\}+\mathcal{O}(r^{2+\alpha}).
\end{split}
\en
Note that both $\tilde{w}_{D}$ and $\tilde{w}_{N}$ are subject to the decay of order $\mathcal{O}(r^{2+\alpha})$ near the corner.
 Letting $r\rightarrow0$ and using the linear independence of the $\sin$ and $\cos$ functions, we get the relations (see Section \ref{sec:3-c} for the proof in the more complicated case of circular cones)
  $$c_D=c_N=0,\qquad\quad  d_{D,j}=d_{N,j}=0\quad\mbox{if}\quad j\pi/\omega< 2.$$
  Hence, the lowest order term of $v$ near $O$ takes the form
\ben
d_{D}\, r^{ 2} \,\sin2\theta+q_{ 2,D}=d_{N}\, r^{ 2} \,\cos2\theta+q_{ 2,N}=:q_{ 2}\in \PP_{ 2},
\enn where $d_D=d_N=0$ if $\omega\neq \pi/2, 3\pi/2$.
Moreover, the polynomial $q_{ 2}$ must satisfy $q_{ 2}=\partial_{\nu}q_{ 2}=0$ on $\partial \K$ and the equations
\ben
\Delta q_{ 2}=\Delta q_{ 2,D}=\Delta q_{ 2,N}=p_0\in\PP_0,\qquad \Delta^2 q_2=0\qquad\mbox{in}\quad \K.
\enn Making use of Proposition \ref{P5} in the Appendix, we then get $q_{ 2}\equiv0$, so that $p_0=0$. This implies that $v\in\Lambda_0^{ 2,\alpha}$. Finally, the relation $u_1(O)=0$ follows from (\ref{eq:27}) and the definition of $p_0$ in
  (\ref{eq:5}).

\textbf{Step 4.} For any $m\in \N$, we show via induction that, for $\alpha>0$ sufficiently small,
\be\label{eq:6}
 f\in \Lambda_{1-m}^{\,0,\alpha},\quad  v\in \Lambda_{1-m}^{\, 2,\alpha}, \quad \nabla^j u_1(O)=0\quad\mbox{for all}\quad j\in \N_0,\; j\leq m-1.
\en Note that the case $m=1$ has been covered by Step 2, and the last equality in (\ref{eq:6}) means that $\partial_{x_1}^{j_1}\partial_{x_2}^{j_2} u_1(O)=0$ for all $j_1,j_2\in \N_0$ such that $j_1+j_2=j$.
 Assuming the induction hypothesis that the relations in
(\ref{eq:6}) hold for some $m>1$, we have to show that
\be\label{eq:7}
 f\in \Lambda_{-m}^{\,0,\alpha},\quad  v\in \Lambda_{-m}^{\, 2,\alpha}, \quad \nabla^m u_1(O)=0.
\en
Denote by $u_{1,m}\in \PP_m$ the homogeneous Taylor polynomial of degree $m$ of $u_1$ at $O$.
By the last relation in (\ref{eq:6}), we have  $u_{1,j}\equiv 0$ for all $j\leq m-1$.

From the induction hypothesis and the assumption on $q$ it follows that
\ben
\chi\,\nabla^l(qu)\in \Lambda_{1-m}^{\, 2,\alpha}\subset \Lambda_{-m}^{\,0,\alpha},\qquad k^2\chi\, \nabla^l(h\,u_1)\,\in \Lambda_{1-m}^{\,0,\alpha}.
\enn This implies that the right hand side can be split into
\ben
f=\chi p_m+f_m\in \Lambda_{1-m}^{\,0,\alpha}
\enn with
\ben
 p_m:=k^2\, \nabla^lh(O)\,u_{1,m},\quad \chi p_m\in \Lambda_{1-m}^{\,0,\alpha},\quad f_m:=f-\chi p_m\in \Lambda_{-m}^{\,0,\alpha}.
\enn
By Proposition \ref{P5} in the Appendix we see that $\Delta p_m=k^2\nabla^lh(O)\Delta u_{1,m}=0 $.

Repeating the arguments in Step 2 and applying Proposition \ref{P5-2d} with $\kappa=m$, we find that $v\in \Lambda^{2,\alpha}_{1-m,D}\cap \Lambda^{2,\alpha}_{1-m,N}$ takes the form
\be\label{eq:23}\begin{split}
v=&\,\chi\left\{q_{m+ 2,D}+c_D\,r^{m+ 2}\left\{ \ln r \sin (m+2)\theta+\theta\cos(m+2)\theta \right\}\right\}\\
&+\,\chi\,\sum_j d_{D,j}\, r^{ j\pi/\omega} \,\sin[(j\pi/\omega)\theta]+\tilde{w}_D\\
= &\,\chi\left\{q_{m+ 2,N}+c_N\,r^{m+ 2}\left\{ \ln r \cos (m+2)\theta-\theta\sin(m+2)\theta \right\}\right\}\\
&+\,\chi\,\sum_j d_{N,j}\, r^{ j\pi/\omega} \,\cos[(j\pi/\omega)\theta]+\tilde{w}_N,
\end{split}
\en
for some
 $\tilde{w}_{D/N}\in \Lambda^{2,\alpha}_{-m,D/N}$, $c_{D/N}\in\C$, $d_{D/N,j}\in \C$ and  $q_{m+ 2,D/N}\in \PP_{m+2}$ satisfying
 $\Delta q_{m+ 2,D/N}=p_m$. The two sums in (\ref{eq:23}) are taken over $j\in \N$ such that
 $$j\pi/\omega\in(1+\alpha+m, 2+\alpha+m),\quad\mbox{or equivalently,}\quad j\pi/\omega\in(1+\alpha+m, 2+m].$$
It is easy to observe that
$\tilde{w}_{D/N}=\mathcal{O}(r^{2+m+\alpha})$ as $r\rightarrow0$.
Hence, it follows from (\ref{eq:23}) by letting $r\rightarrow0$ that
\ben
c_D=c_N=0,\quad d_{D,j}=d_{N,j}=0\quad\mbox{if}\quad j\pi/\omega< m+2\, ;
\enn see again the proof of Lemma \ref{lem-cone} for the details.
Therefore, the lowest order term $q_{m+2}$ of $v$ near $O$ belongs to $\PP_{ m+2}$ and satisfies
\ben
&\Delta q_{ m+2}=\Delta q_{m+2,D/N} =p_{m}\in\PP_{ m},\quad \Delta^2q_{m+2}=\Delta p_m=0\quad\mbox{in}\quad \K\, \\
&q_{ m+2}=\partial_{\nu}q_{ m+2}=0\quad\mbox{on}\quad \partial \K.
\enn
Using Proposition \ref{P6} in the Appendix we arrive at $q_{ m+2}\equiv0$. Consequently, it follows that  $p_{m}\equiv0$ and $u_{1,m}\equiv 0$ which implies the relations in (\ref{eq:6}).

\textbf{Step 5.} We have proved that $\nabla^j u_1(O)=0$ for all $j\in \N_0$ in the previous step. Hence, $u_1\equiv0$ in $B_1$ due to the analyticity. Finally, the vanishing of $u_2$ follows from the unique continuation for elliptic equations; see e.g.  \cite[Chapters 3.2 and 3.3]{Isakovbook} for a proof based on Carleman estimates. This finishes the proof of Lemma \ref{lem:corner-2D}.
\end{proof}
%\begin{rem} Our results in 2D
% remain valid if we replace the assumption (a) for $l=1$ by the following weaker condition:
%\be\label{assumption}
%q\in C^{1,s}(\overline{B_\epsilon\cap D}),\quad q(O)=1,\quad \nabla q(O)\neq 0.
%\en
%Compared to (\ref{q}), there is no requirement of the continuity of $q$ on $B_\epsilon\cap \partial D$ in (\ref{assumption}). For example, the function $q|_{\overline{D}}=1+\sin(x_1+x_2)$ satisfies (\ref{assumption}), but it does not fulfill the assumption (a) for $l=1$. To prove Lemma \ref{lem:corner-2D} under the condition (\ref{assumption}), we set $v=\chi(u_1-u_2)$ for some cut-off function supported in $\K$ and consider the boundary value problem
%\ben
%&&\Delta v=-k^2\,q v+k^2\chi\,(1-q)u_1-[\Delta,\chi](u_1-u_2)=:f\quad\mbox{in}\quad \K\cap B_1,\\
%&& v=\partial_\nu v=0 \quad\mbox{on}\quad \partial \K\cap B_1.
%\enn
%By (\ref{assumption}), it holds that $(1-q)u_1=q_1 u_{1,m}+\mathcal{O}(r^{m+1})$, where $q_1\in\PP_1$ and $u_{1,m}\in\PP_m$ ($m\in\N_0$) denotes again the lowest order term in the Taylor expansion of $u_1$ at $O$.
%split the right hand side of () into the form
%\end{rem}
\section{Edges in 3D always scatter}\label{3D-edges}
This section is devoted the proof of Lemma \ref{lem} for a polyhedron with flat surfaces.
Consider an infinite wedge domain $\W=\K\times\R$ in $\R^3$, where the notation $\K$ still stands for a sector with the opening angle $\omega\in (0,2\pi)\backslash\{\pi\}$. For simplicity we write $x'=(x_1,x_2)$ so that $x=(x',x_3)\in \R^3$. Analogously, the origin $O\in \R^3$ can be written as $O=(O',0)$ where $O'=(0,0)\in \R^2$.
Let $U_{a}=\{x\in \R^3: x_1^2+x_2^2<1, |x_3|<a\}$ be a cylinder of height $2a$ for some $a>0$. Then  $O\in \partial \W\cap U_1$ is an interior edge point. Let $\Delta=\Delta_x$ and $\Delta_{x'}$ be the three and two dimensional Laplace operators with respect to the variables $x$ and $x'$, respectively.  Suppose that $q\in C^{\,l,s}(\overline{\W\cap U_1})$ for some $s\in(0,1)$ and $l\in \N_0$ and that $q\equiv 1$ in $\W^e\cap U_1$.
As the counterpart of (\ref{eq:1.0}) in 3D, we consider the problem
\be\begin{split}\label{eq:11}
&\Delta u_1+ k^2 u_1=0,\qquad \Delta u_2+ k^2 q u_2=0 \quad \mbox{in}\quad U_{1},\\
&\partial_\nu^{j}(u_1-u_2)=0\quad \mbox{on}\quad \partial \W\cap U_{1},\; j=1,2,\cdots,l+1.
\end{split}
\en
The analogue of Lemma \ref{lem:corner-2D} in a wedge domain is formulated as follows.
\begin{lem}\label{lem:edge-3D} Assume that $q$ satisfies the assumption (a) with $D:=\W\cap U_{1}$ near the edge point $O$.
Let $u_1,u_2\in H^2(U_{1})$ be a solution pair to (\ref{eq:11}). Then $u_1=u_2\equiv0$ in $U_{1}$.
\end{lem}

Based on Lemma \ref{lem:edge-3D} one can prove that an edge with an arbitrary opening angle $\omega\in(0,2\pi)\backslash\{\pi\}$ scatters every incident wave non-trivially (see Section \ref{proof}). Below we extend the arguments for proving Lemma \ref{lem:corner-2D} to a wedge domain by using partial Fourier transform. Lemma \ref{lem} in the case of a polyhedron with flat surfaces is a direct consequence of Lemma \ref{lem:edge-3D}.

\begin{proof}
By Proposition \ref{P1},  $u:=u_2-u_1\in C^{\,l,\alpha}(\overline{U}_{1})\cap H^{l+2}(U_{1})$ for all $\alpha\in[0,1)$. To prove the lemma, we set $h:=1-q$ and $v(x):=\chi(x')\varphi(x_3) \nabla_x^l u$ where $\chi\in C_0^\infty(\overline{\K})$ is the cut-off function introduced in the proof of Lemma \ref{lem:corner-2D} and $\varphi\in C_0^\infty(-1,1)$ satisfies $\varphi\equiv1$ in $(-1/2,1/2)$. Then $v\in C^{\,1,\alpha}(\overline{\W})\cap H^{2}(\W)$ is a solution to the inhomogeneous Laplace equation (cf. (\ref{eq:2}))
\be\label{eq:12}\begin{split}
&\Delta v=-k^2\chi\varphi\,\nabla_x^l(qu)+k^2\chi\varphi\,\nabla_x^l(hu_1)-[\Delta,\chi\varphi](\nabla_x^l u)=:f_0\quad\mbox{in}\quad \W,\\
&v=\partial_\nu v=0 \quad\mbox{on}\quad \partial \W.
\end{split}
\en
Introduce the partial Fourier transform
\ben
\F_{x_3\rightarrow\xi}(v(x',x_3))=\F v(x',\xi):=\frac{1}{\sqrt{2\pi}}\int_{\R} v(x',x_3)\,e^{i x_3\xi}\;d x_3,\quad \xi\in\R
\enn
and set
\ben
w_0:=\varphi\nabla_x^l u,\qquad w(x',\xi):=\chi(x')\F w_0(x',\xi)=\F v(x',\xi).
\enn
Applying the  partial Fourier transform
to (\ref{eq:12}),
we obtain a Cauchy problem for the two-dimensional Laplace equation in the infinite sector $\K$ depending on the parameter $\xi\in \R$:
\be\label{eq:13}\begin{split}
&\Delta_{x'} w(x',\xi)=\F f_0(x',\xi)+\xi^2 w(x',\xi)=:f(x',\xi)\quad&&\mbox{in}\quad \K, \\
&w(\cdot,\xi)=\partial_{\nu}w(\cdot,\xi)=0\quad&&\mbox{on}\quad \partial \K.
\end{split}\en
Note that the right hand side $f$ is analytic in $\xi$ for any fixed $x'\in\R^2$. Moreover, for all $\xi\in \R$, we have for $\alpha\leq s$ that
\ben
w(\cdot,\xi)\in C^{1,\alpha}(\overline{\K})\cap H^2(\K),\quad f(\cdot,\xi)\in C^{0,\alpha}(\overline{\K})\subset \Lambda^{0,\alpha}_1(\K)
\enn and
\ben
f(\cdot,\xi)-f(O',\xi)\in \Lambda^{0,\alpha}_0(\K).
\enn
Applying Step 2 in the proof of Lemma \ref{lem:corner-2D} to  (\ref{eq:13}) yields $w(\cdot,\xi)\in \Lambda^{2,\alpha}_{1,D}(\K)\cap \Lambda^{2,\alpha}_{1,N}(\K)$, if $0<\alpha<s$ is chosen sufficiently small.
 Further, by arguing analogously to Step 3 in the proof of Lemma \ref{lem:corner-2D} we obtain
\ben
f(\cdot,\xi)\in \Lambda^{0,\alpha}_0,\quad f(O',\xi)=0,\quad
w(\cdot,\xi)\in \Lambda^{2,\alpha}_{0,D}(\K)\cap \Lambda^{2,\alpha}_{0,N}(\K),\quad\forall\;\xi\in\R.
\enn Together with (\ref{eq:13}) this leads to $\F f_0(O',\xi)=0$ for all $\xi\in \R$ and thus $f_0(O',x_3)=0$, $x_3\in \R$. In view of the definition of $f_0$ on the right hand side of (\ref{eq:12}) we see that
\ben
0=f_0(O',x_3)=k^2\nabla_x^l h(O',x_3)\;u_1(O',x_3)\quad\mbox{for all}\quad |x_3|<1/2,
\enn
where we have used the fact that $\nabla^j u=0$ on $\partial \W\cap U_1$ for all $j=0,1,\cdots,l+1$.
By the continuity of $\nabla_x^l h(O',x_3)$ near $x_3=0$ and using the assumption $\nabla_x^l h(O)\neq 0$, we get
$u_1(O',x_3)\equiv 0$ for $|x_3|$ sufficiently small. Further, $u_1(O',x_3)\equiv 0$ for all $x_3\in \R$ by the analyticity, and in particular $u_1(O)=0$.

For $\beta=(\beta_1,\beta_2)\in \N_0^2$, let $|\beta|=\beta_1+\beta_2$.
Denote by $u_{1,j}(\cdot, x_3)$, $j\in \N_0$, the homogeneous Taylor expansion of degree $j$ of $u_1(\cdot,x_3)$ at $x'=O'$ which takes the form
\be\label{eq:17}
u_{1,j}(x',x_3)=\sum_{|\beta|=j} c_{\beta,j}(x_3)\; x'^\beta,\quad  x'^\beta=x_1^{\beta_1}x_2^{\beta_2}.
\en For some $m\in \N$, $m>1$, we make the induction hypothesis that
\be\label{eq:15}\begin{split}
&w(\cdot,\xi)\in \Lambda_{1-m}^{2,\alpha}(\K),\quad f(\cdot,\xi)\in \Lambda_{1-m}^{0,\alpha}(\K)\quad&&\mbox{for all}\;\xi\in \R,\\
&u_{1,j}\equiv0\quad&&\mbox{for all}\; j<m.
\end{split}\en
We need to prove that
\ben
w(\cdot,\xi)\in \Lambda_{-m}^{2,\alpha}(\K),\quad f(\cdot,\xi)\in \Lambda_{-m}^{0,\alpha}(\K)\quad\mbox{for all}\;\xi\in \R,\quad u_{1,m}\equiv0.
\enn
Note that the relations in (\ref{eq:15}) for $m=1$ have been verified in the previous step and that the last relation in (\ref{eq:15}) implies that, for all $x_3\in \R$,
\be\label{eq:25}
u_1(x',x_3)-u_{1,m}(x',x_3)=\mathcal{O}(|x'|^{m+1})\quad\mbox{as}\quad |x'|\rightarrow 0.
\en

The right hand side of the equation in (\ref{eq:13}) takes the form (cf. (\ref{eq:12}))
\be\label{eq:14}
f=k^2\chi \F(\varphi\,\nabla_x^l(hu_1))-k^2\chi \F(\varphi\,\nabla_x^l(qu))-\F([\Delta,\chi\varphi](\nabla_x^l u))
+\xi^2 w.
\en
Obviously, $\F([\Delta,\chi\varphi](\nabla_x^l u))(\cdot, \xi)\in \Lambda_{-m}^{0,\alpha}(\K)$.
Using the induction hypothesis on $w$ and the regularity of $q$ it can be readily checked that, for all $\xi\in \R$,
\ben
\xi^2 w(\cdot,\xi)\in \Lambda_{1-m}^{2,\alpha}(\K)\subset \Lambda_{-m}^{0,\alpha}(\K), \quad \F(\varphi\,\nabla_x^l(qu))(\cdot,\xi)\in \Lambda_{-m}^{0,\alpha}(\K).
\enn
To estimate the first term on the right hand side of (\ref{eq:14}), we use (\ref{eq:25}) and the assumption on $q$ to derive the decompositions
\ben
h(x',x_3)\,u_1(x',x_3)&=&h(O',x_3)\,u_{1,m}(x',x_3)+\mathcal{O}(|x'|^{l+m+1}),\\
\varphi(x_3)\nabla_x^l[h(x',x_3)u_1(x',x_3)]&=&\varphi(x_3)\nabla_x^l[h(O',x_3)]\;u_{1,m}(x',x_3)+\varphi(x_3)\mathcal{O}(|x'|^{m+1})
\enn
as $|x'|\rightarrow0$.
Taking the partial Fourier transform gives
\be\label{eq:16}
\F\left(\varphi\,(\nabla_x^l(hu_1)- \nabla_x^l h(O',x_3)\; u_{1,m} )\right) (\cdot,\xi)\in \Lambda_{-m}^{0,\alpha}(\K)\quad\mbox{for all}\quad \xi\in \R.
\en
Now, combining (\ref{eq:17}), (\ref{eq:14}) and (\ref{eq:16}) we see that
\ben
f(\cdot,\xi)-\chi k^2 \hat{p}_m(\cdot,\xi)\in \Lambda_{-m}^{0,\alpha}(\K),
\enn where $\hat{p}_m(\cdot,\xi)\in \PP_m$ is defined as
\ben
\hat{p}_m(x',\xi)=\sum_{|\beta|=m} \tilde{c}_\beta(\xi)\,x'^\beta,\quad
\tilde{c}_\beta(\xi):=
\F\left(\varphi(x_3)\nabla_x^l[h(O',x_3)]\;\;c_{\beta,m}(x_3)\right)(\xi).
\enn Note that the coefficients $\tilde{c}_\beta$ are analytic on $\R$ and belong to $L^1(\R)$, since the functions
$\varphi\nabla_x^l[h(O',x_3)]c_{\beta,m}$
are continuous and have a compact support on $\R$.
Applying the arguments in the proof of Lemma \ref{lem:corner-2D}, we may conclude that the lowest order term $Q_{m+2}(\cdot,\xi)$ of $w(\cdot,\xi)$ near the corner of $\K$ belongs to $\PP_{ m+2}$ and satisfies the Cauchy problem
\be\label{eq:18}\begin{split}
&\Delta_{x'} Q_{m+2}(\cdot,\xi)=\hat{p}_m(\cdot,\xi)\quad&&\mbox{in}\quad \K,\\
&Q_{m+2}(\cdot,\xi)=\partial_{\nu}Q_{m+2}(\cdot,\xi)=0\quad&&\mbox{on}\quad \partial \K
\end{split}
\en
for all $\xi\in\R$. %We may represent $\hat{q}_{m+2}$ as
%\ben
%\hat{q}_{m+2}(x',\xi)=\sum_{|\beta|=m+2}\tilde{d}_{\beta}(\xi)\,x'^{\beta},
%\enn
%so that $\Delta_{x'} \hat{q}_{m+2}(\cdot,\xi) =\sum_{|\beta|=m}C(\beta)\tilde{d}_{\beta}(\xi)\,x'^{\beta}.$

Since $\F p_m (x', \cdot)\in L^1(\R)$,  its inverse Fourier transform is given by
%Taking the inverse Fourier transform in (\ref{eq:18}), we arrive at
%\ben
%&\Delta_{x'}\; q_{m+2}(\cdot,x_3)=p_{m}(\cdot,x_3)\;\mbox{in}\quad \K,\quad
%&q_{ m+2}(\cdot,x_3)=\partial_{\nu}q_{ m+2}(\cdot,x_3)=0\quad\mbox{on}\quad \partial \K,
%\enn where $p_{m}\in \PP_{m}$ is given by
\be\label{eq:28}
p_m(x',x_3)=\varphi(x_3)\nabla_x^l[h(O',x_3)]\;u_{1,m}(x',x_3).
\en  Recalling the induction hypothesis that $u_{1,j}(x',x_3)\equiv 0$ for all $0\leq j<m$ (see (\ref{eq:17}) and (\ref{eq:15})), we get
\ben
u_1(x',x_3)= u_{1,m}(x', x_3)+\mathcal{O}((|x'|+|x_3-t|)^{m+1}),
\enn as $|x'|\rightarrow 0$, $x_3\rightarrow t$ for all $t\in\R$. Hence,  $u_{1,m}$ coincides with the lowest order term $U_{1,m}$ in the Taylor expansion of $u_1$ at $(x', t)\in \R^3$. As a consequence of Proposition \ref{P7} (iii), it holds for all $x_3\in\R$ that $\Delta_{x'} u_{1,m}(x',x_3)=\Delta_x U_{1,m}\equiv0$ and thus
 \be\label{eq:29}
 \Delta_{x'}(p_m(x', x_3))=\varphi(x_3)\nabla_x^l[h(O',x_3)]\;\Delta_{x'} u_{1,m}(x',x_3)\equiv0.
  \en
Taking the partial Fourier transform of (\ref{eq:29}) with respect to $x_3$ gives
 \ben
 \Delta_{x'}\hat{p}_m(x', \xi)\equiv0\quad \mbox{for all}\quad \xi\in\R.
 \enn
 Together with (\ref{eq:18}) this implies that $Q_{m+2}(\cdot, \xi)$ is a biharmonic function with vanishing Dirichlet and Neumann data on $\partial \K$.
 Now, applying Proposition \ref{P6} to $Q_{ m+2}$ gives the relations $Q_{m+2}(\cdot,\xi)=\hat{p}_m(\cdot,\xi)\equiv 0$ for all $\xi\in\R$, which further result in
 \ben
f(\cdot,\xi)\in \Lambda_{-m}^{0,\alpha}(\K),\quad w(\cdot,\xi)\in \Lambda_{-m}^{2,\alpha}(\K),\quad
p_m\equiv 0.
\enn
Since  $\nabla_x^l[h(O',x_3)]\neq 0$ in a neighborhood of $x_3=0$, it follows from (\ref{eq:28}) that
  $u_{1,m}(x',x_3)\equiv0$ in a neighborhood of the plane $x_3=0$ in $\R^3$.
  Hence, $u_{1,m}\equiv0$ in $\R^3$ due to the analyticity.
    This proves the relation $\nabla^m u_1 =0$ at $O$.
  Since $m\in \N_0$ is arbitrary, the relation $u_1\equiv 0$ follows.
 Finally, we obtain $u_2\equiv 0$ by unique continuation.
   The proof of Lemma \ref{lem:edge-3D} is thus complete.

\end{proof}

\section{Curvilinear polygons and polyhedra always scatter}\label{sec: curvilinear}

In this section we shall adapt the arguments in Sections \ref{sec:3} and \ref{3D-edges} to the case of a curvilinear polygon or polyhedron. Lemma \ref{lem} in the cases (i) and (ii) can be equivalently stated as

\begin{lem}\label{lem:curvilinear} Let $D$ be a bounded curvilinear polygon or polyhedron and let
 the potential $q$ satisfy the assumption (a) near a  corner or edge point $P\in\partial D$.
 For $\epsilon>0$ sufficiently small, let $\Gamma_\epsilon=B_\epsilon(P)\cap \partial D$ be a sub-boundary
 of $\partial D$ such that $P\in \Gamma_\epsilon$.
% Let $\Gamma\subset \partial D$ be a sub-boundary such that $P\in \Gamma$.
  If the solution pair $u_j\in H^2(B_\epsilon(P))$ ($j=1,2$) solves the coupling problem (\ref{eq:30}),
then $u_1=u_2\equiv 0$.
\end{lem}
\begin{proof} For brevity we only indicate the changes that are necessary to reduce the case of a curvilinear domain
to a sector or wedge domain. %Below we only present the proof for curvilinear polygons, because the case of a curvilinear polyhedron can be treated analogously.
We start with the same argument as in the proof of Lemmas \ref{lem:corner-2D} and \ref{lem:edge-3D} by
choosing an appropriate cut-off function $\chi$ in a neighbourhood of $P$ in $D$. Consequently, the function
$v:=\chi(x)\nabla_x^l (u_1-u_2)$ satisfies the boundary value problem (cf. (\ref{eq:2}) and (\ref{eq:12}))
\be\label{eq:19}
\Delta\, v=f\quad\mbox{in}\quad D\cap B_\epsilon(P),\quad
v=\partial_\nu v=0 \quad\mbox{on}\quad \Gamma_\epsilon
\en for some H\"older continuous function $f$ supported in a neighborhood of $P$ in $D$.
 Denote by $y=\Psi(x)$, $y=(y_1,y_2,\cdots,y_n)$, $x=(x_1,x_2,\cdots,x_n)$,
 the diffeomorphism specified in Definitions \ref{def-2d} and \ref{def-2d} mapping a curvilinear domain near $P$
 to a sector or wedge domain with flat boundaries. For notational convenience we write $\mathcal{U}=\K$ in two dimensions and $\mathcal{U}=\W=\K\times \R$ in three dimensions.
 Under the transformation
 \ben
 \tilde{v}(y)=v(\Psi^{-1}(y)),\quad \tilde{f}(y)=f(\Psi^{-1}(y)),\quad y\in \R^n,
 \enn we have
 \be\begin{split} \label{eq:20}
 &\Delta_y \tilde{v}=\Delta_x v(\Psi^{-1}(y))-g_{\tilde{v}}(y)=\tilde{f}(y)-g_{\tilde{v}}(y)\quad&&\mbox{in}\quad  \mathcal{U},\\
&\tilde{v}=\partial_\nu\tilde{v}=0\quad&&\mbox{on}\quad\partial\mathcal{U}.
 \end{split}\en
 where
  \ben
 &&g_{\tilde{v}}(y):=\sum_{i,j=1}^n [a_{ij}(y)-\delta_{ij}]\; \frac{\partial^2 \tilde{v}}{\partial y_j\partial y_i}+\sum_{i=1}^n b_i(y) \frac{\partial \tilde{v}}{\partial y_i},\\
 &&a_{ij}(y):=\left(\nabla_x\, y_i(x)\cdot \nabla_x\, y_j(x)\right)|_{x=\Psi^{-1}(y)},\\
 &&b_i(y):=(\Delta_x\, y_i(x))|_{x=\Psi^{-1}(y)}.
 \enn Here $\delta_{ij}$ is the Kronecker delta symbol.
Compared to the right hand sides of (\ref{eq:2}) and (\ref{eq:12}), the term $-g_{\tilde{v}}$ in (\ref{eq:20}) is additional.
Since $\nabla \Psi=I$ and $\Psi$ is of $C^2$-smoothness, it holds that
\ben
a_{ij}(y)-\delta_{ij}=\mathcal{O}(|y|),\qquad b_i(y)=\mathcal{O}(1)\quad\mbox{as}\; y\rightarrow O,\quad i,j=1,2,\cdots,n.
\enn
Hence, if $\tilde{v}\in \Lambda^{2,\alpha}_{1-m}(\mathcal{U})$ for some $m\in \N_0$, then
it must hold that $g_{\tilde{v}}\in \Lambda^{0,\alpha}_{-m}(\mathcal{U})$ because
\ben
[a_{ij}(y)-\delta_{ij}]\; \frac{\partial^2 \tilde{v}}{\partial y_j\partial y_i}\in \Lambda^{0,\alpha}_{-m}(\mathcal{U}),
\quad b_i(y) \frac{\partial \tilde{v}}{\partial y_i}\in \Lambda^{0,\alpha}_{-m}(\mathcal{U}),
\enn for all $i,j=1,\cdots,n$.
Proceeding by induction on $m$, suppose that $f\in \Lambda^{2,\alpha}_{1-m}(B_\epsilon(P)\cap D)$ takes the form
\ben
f=\chi p_m+f_m,\quad\chi p_m\in \Lambda_{1-m}^{0,\alpha}(D\cap B_\epsilon(P)),\quad f_m\in \Lambda_{-m}^{\,0,\alpha}( D\cap B_\epsilon(P))
\enn
for some $p_m\in\PP_{ m}$.
Then by the assumptions on $\Psi$ the transformed function $\tilde{f}$ can be written as
\ben
\tilde{f}=\tilde{\chi} q_m+g_m,\quad \tilde{\chi}q_m\in \Lambda_{1-m}^{\,0,\alpha}(\mathcal{U}),\quad g_m\in \Lambda_{-m}^{\,0,\alpha}(\mathcal{U})
\enn for some $q_m\in\PP_{m}$. Further, the relation $q_m\equiv 0$ then implies the vanishing of $p_m$ and also
of the $m$-th order terms in the Taylor expansion of $u_1$ at $P$.
 Applying the arguments in the proof of Lemmas \ref{lem:corner-2D} and \ref{lem:edge-3D} to the equation
(\ref{eq:20}), we successively obtain $q_m\equiv0$ for all $m\in \N_0$, which implies $u_1=u_2\equiv 0$.
\end{proof}

\section{Circular cones always scatter}\label{sec:cone}
This section is concerned with the scattering problems corresponding to a penetrable obstacle with circular conic corners on the boundary. We first present the solvability of the Laplace equation in a three-dimensional cone and then verify Lemma \ref{lem} in the case (iii).
\subsection{Solvability of the Laplace equation in a circular conic domain}\label{Regularity-C}
Let $\CC$ be the infinite circular cone introduced in Section \ref{cone}. For $\beta\in\R$, $\kappa\in\N_0$ and $\alpha\in[0,1)$,  we define the
 weighted spaces $V_\beta^{\kappa}(\CC)$, $V_{\beta,D/N}^{\kappa}(\CC)$, $\Lambda_\beta^{\kappa,\alpha}(\CC)$
 and $\Lambda_{\beta,D/N}^{\kappa,\alpha}(\CC)$ in the same way as in Section \ref{Regularity}, where only
 the sector $\K\subset\R^2$ is replaced with
  the cone $\CC\subset\R^3$ and $r$ denotes the distance of $x$ to the conic point $O$. In this section we denote by
  $\Delta_D$ resp. $\Delta_N$ the operator of the Dirichlet resp. Neumann problem corresponding to the inhomogeneous Laplace equation with the homogeneous boundary condition on $\partial \CC$ acting on the spaces $V_{\beta,D/N}^{\kappa}(\CC)$ and $\Lambda_{\beta,D/N}^{\kappa,\alpha}(\CC)$.
 Consider the Dirichlet and Neumann boundary value problems
\be\label{DNP}
\Delta_D\, u=f, \quad\quad
\Delta_N\, u=f \quad\mbox{on}\quad \overline{\CC}.
\en
Using spherical coordinates we may rewrite the Laplace operator as
\ben
\Delta=\frac{1}{r^2}\left\{(r \frac{\partial}{\partial r})^2+r\frac{\partial}{\partial r}+\hat{\Delta} \right\},\quad\hat{\Delta}:=\frac{1}{\sin\theta}\left(\frac{\partial}{\partial \theta} \sin\theta \frac{\partial}{\partial \theta} +
\frac{1}{\sin\theta} \frac{\partial ^2}{\partial \varphi^2}\right)
\enn
where $\hat{\Delta}$ is the Beltrami operator defined on $\s^2$.
To study the solvability of the boundary value problems (\ref{DNP}) in the weighted Sobolev spaces $V_\beta^{\kappa}(\CC)$ and H\"older spaces $\Lambda_{\beta}^{\kappa,\alpha}(\CC)$, we shall apply Kondratiev's method \cite{Kondratiev} by looking for solutions of the homogeneous problems  (\ref{DNP}) (i.e., $f=0$) in the form $u(x)=r^\lambda V(\hat{x})$ with $\hat{x}=x/r\in\s^2$; cf. \cite{NP} and \cite{KMR}. Then $V$  satisfies the eigenvalue problem
\be\label{eigenvalue}\begin{split}
&\hat{\Delta} V+\lambda  (\lambda+1) V=0&&\quad\mbox{in}\quad \Omega:=\s^2\cap \CC,\\
&V=0\quad\mbox{or}\quad \partial_\nu V=0&&\quad\mbox{on}\quad \partial \Omega.
\end{split}
\en
The Dirichlet and Neumann eigenvalues of (\ref{eigenvalue}), $\lambda_{D,j}$ and $\lambda_{N,j}$ ($j\in \Z\backslash\{0\}$), counted with their finite multiplicities, form a discrete set in $\R$. Further, there are corresponding orthogonal (in $L^2(\Omega)$) sequences of eigenfunctions $V_{j,D}$ and $V_{j,N}$ (see e.g. \cite[Chapter 2]{KMR}).

Below we present a more explicit description of the eigenvalues and eigenfunctions in our case of a circular cone. For this purpose we need the definition of
Legendre functions and spherical harmonic functions.
 For $\lambda \in \R$, denote by $P_\lambda$ the Legendre function of first kind satisfying the Legendre differential equation
\be\label{Legendre}
\frac{d}{dt}\left[(1-t^2)\frac{d f}{dt}   \right]+\lambda (\lambda+1)f=0.
\en
By $P_\lambda^m$ ($m\in\N_0$) we denote the associated Legendre functions of the first kind defined via
\ben
P_\lambda^m(t):=(1-t^2)^{m/2}\frac{d^m P_\lambda(t)}{dt^m},\quad m=0,1,\cdots,n,
\enn
which satisfy the associated Legendre differential equations
 \be\label{Legendre-A}
\frac{d}{dt}\left[(1-t^2)\frac{d f}{dt}\right]+\left[\lambda (\lambda+1)-\frac{m^2}{1-t^2}\right]f=0.
\en
Recall that the normalized spherical harmonic functions of order $n\in\N_0$ are defined by
\be\label{spherical-harmonics}
Y_n^m(\theta,\varphi):=\sqrt{\frac{2n+1}{4\pi}\frac{(n-|m|)\!}{(n+|m|)\!}}\,P_n^{|m|}(\cos\theta)\,e^{im\varphi}
\en
for all $m=-n,\cdots,n$.
By \cite{DP}, $\lambda\in \R$ is an eigenvalue to the Dirichlet resp. Neumann boundary value problem (\ref{eigenvalue}) if and only if there exists some $m\in \Z$ such that $P_\lambda^{|m|}(\cos \omega)=0$ resp.  $(P_\lambda^{|m|})'(\cos \omega)=0$,  with the associated eigenfunction $V=P_\lambda^{|m|}(\cos \theta) e^{im\varphi}\in C^\infty(\overline{\Omega})$. In the special case that $\lambda=n\in \N$ and $|m|\leq n-1$, the eigenfunction
$V=P_n^{|m|}(\cos \theta) e^{im\varphi}$ is a spherical harmonic function of order $n$ and $r^n V\in \PP_n$ is a homogeneous polynomial of order $n$. Note that Dirichlet and Neumann eigenvalues  may coincide. For instance, if $(P_2^0)'(\cos\omega)=P_2'(\cos\omega)=0$, then $P_2^1(\cos\omega)=\sin\omega (P_2^0)'(\cos\omega)=0$, implying that $\lambda=2$  is both a Dirichlet and Neumann eigenvalue.
Since $P^m_{\lambda}=P^m_{-\lambda-1}$, we have
\ben
\lambda_{D/N,-j}=-\lambda_{D/N,j}-1,\quad \lambda_{D,1}>0,\,\lambda_{N,1}=0,\, \lambda_{N,-1}=-1.
\enn

Below we state the solvability results for the Laplace equation in the weighted spaces $V_\beta^{2}(\CC)$ and $\Lambda_\beta^{2,\alpha}(\CC)$.

\begin{prop}\label{P2C}(\cite[Chapter 3, Theorem 5.1]{NP})
 The operator $\Delta_{D/N}:  V_{\beta, D/N}^{ 2}(\CC)\rightarrow V_{\beta}^{0}(\CC)$ is an isomorphism if $ 1/2-\beta\neq \lambda_{D/N,j}$ for all $j\in \Z\backslash\{0\}$.
\end{prop}
\begin{prop}\label{P3C}(\cite[Chapter 3, Theorem 6.11]{NP})
 The operator $\Delta_{D/N}:  \Lambda_{\beta, D/N}^{ 2,\alpha}(\CC)\rightarrow \Lambda_{\beta}^{0,\alpha}(\CC)$ is an isomorphism if $ 2+\alpha-\beta\neq \lambda_{D/N,j}$ for all $j\in \Z\backslash\{0\}$.
\end{prop}

\begin{prop}\label{P4C}(\cite[Chapter 3, Theorem 6.9]{NP})
Let $\gamma_1<\gamma\leq  2$  and assume $ 2+\alpha-\beta\neq \lambda_{D/N,j}$ for $\beta=\gamma,\gamma_1$ and for all $j\in \N$. Moreover, let $f\in \Lambda_{\gamma}^{0,\alpha}(\CC)\bigcap \Lambda_{\gamma_1}^{0,\alpha}(\CC)$ and denote by $v_\beta\in \Lambda_{\beta,D/N}^{ 2,\alpha}(\CC)$ the unique solution of the problem $\Delta_{D/N} v=f\in \Lambda_\beta^{0,\alpha}(\CC)$. Then the relation
\ben
v_{\gamma_1}=v_{\gamma}+\sum_j C_j\,r^{\lambda_{D/N,j}}\,V_{j,\,D/N}(\hat{x}),\quad C_j\in \C
\enn
holds,
where the sum is taken over all $j\in\N$ such that $\lambda_{D/N,j}\in( 2+\alpha-\gamma, 2+\alpha-\gamma_1)$.
\end{prop}

The following is a special case of \cite[Chapter 3, Lemma 5.11]{NP} with additional information in the case of a circular cone.
\begin{prop}\label{Special-solution-C} For $\kappa\in \N_0$,
consider the inhomogeneous problem $\Delta_{D/N} v=p_{\kappa}\in \mathbb{P}_{\kappa}$ on $\overline{\CC}$. There exists a special solution of the form
\be\label{special-solution-C}
v_{D/N}=q_{D/N, \kappa+2}\in \PP_{ \kappa+2}
\en
if $\lambda_{D/N,j}\neq \kappa+ 2$ for all $j\in \N$, and
\be\label{special-solution-2-C}
v=q_{D/N, \kappa+2}+\sum_m C_{D/N,m}\,r^{\kappa+ 2}\left\{ \ln r\, Y_{\kappa+ 2}^{m}(\hat{x})+\psi_{D/N,m}(\hat{x}) \right\}
\en
if $\kappa+ 2$ is a Dirichlet resp. Neumann eigenvalue. In (\ref{special-solution-2-C}), $C_{D/N,m}\in \C$,
$\psi_{D/N,m}\in C^\infty(\overline{\Omega})$ and the sum is taken over all $m\in \Z$ such that $|m|\leq \kappa$ and $P_{\kappa+ 2}^{|m|}(\cos\omega)=0$ in the Dirichlet case and $(P_{\kappa+ 2}^{|m|})'(\cos\omega)=0$ in the Neumann case.
\end{prop}

\begin{proof}
Applying Proposition \ref{P7} (i), we may expand $p_{\kappa}\in \mathbb{P}_{\kappa}$ as
\ben
p_\kappa(r,\theta,\varphi)=r^\kappa \sum_{n,j \in \N_0: n+2j=\kappa}\; \sum_{m=-n}^n a_{n,m}^{(j)}\, Y_n^m(\theta,\varphi).
\enn
Hence, it suffices to prove the proposition for a term of the form
\be\label{Pkappa}
p_\kappa(x)=r^\kappa\; Y_n^{m}(\hat{x})\quad\mbox{for some}\quad 0\leq n\leq \kappa,\; |m|\leq n.
\en One can readily look for a polynomial $q_{\kappa+2}$ to the equation $\Delta q_{\kappa+2}=p_{\kappa}$ in the form
\be\label{qkappa}
q_{\kappa+2}(x)=\zeta\;r^{\kappa+2}\;Y_n^{m}(\hat{x}),\quad \zeta=\frac{1}{(\kappa+2)(\kappa+3)-n(n+1)}.
\en
We first consider the Dirichlet boundary value problem. In the case $\kappa+2\neq \lambda_{D,j}$ for all $j$, we have $P_{\kappa+2}^{|m|}(\cos\omega)\neq0$. Setting
\ben
q_{D,\kappa+2}(x):=q_{\kappa+2}(x)-q_{\kappa+2}(r,\omega,\varphi)P_{\kappa+2}^{|m|}(\cos\theta)/P_{\kappa+2}^{|m|}(\cos\omega),
\enn we obtain the requested polynomial solution. Now we assume that $\kappa+2$ is a Dirichlet eigenvalue of (\ref{eigenvalue}) with the associated eigenfunction $V=Y_{\kappa+2}^{m}$, which implies that $P_{\kappa+2}^{|m|}(\cos\omega)=0$.
 As in \cite[Chapter 3]{NP} we make the ansatz
\be\label{VD}
v_D(r, \hat{x})=c\, r^{\kappa+2}\ln r \;Y_{\kappa+2}^{m}(\hat{x})+r^{\kappa+2}\,W(\hat{x})
\en with an unknown constant $c\in \C$ and an unknown function $W$ to be determined from the Dirichlet boundary value problem
\be\label{DBVP}\begin{split}
\hat{\Delta}W+(\kappa+2)(\kappa+3)W&=-c(2\kappa+5)Y_{\kappa+2}^{m}+Y_{n}^{m}=:F\quad&&\mbox{in}\quad \Omega=\CC\cap \s^2,\\
W&=0\qquad&&\mbox{on}\quad \partial\Omega,
\end{split}\en where the number $n\in \N_0$ is the same as that in (\ref{Pkappa}).
Note that if $W$ solves the previous boundary value problem, then the solution $v_D$ of the form (\ref{VD})
must be a Dirichlet eigenfunction to (\ref{eigenvalue}).
The constant $c$ will be selected such that the right hand side $F$  is orthogonal to  $Y_{\kappa+2}^{m}$ in the $L^2(\Omega)$-sense, i.e.,
\ben
c=\frac{\int_\Omega Y_n^m \overline{Y^m_{\kappa+2}}\,ds}
{(2\kappa+5)\int_\Omega |Y^m_{\kappa+2}|^2\,ds  }.
\enn
Hence the problem (\ref{DBVP}) admits at least one solution by the Fredholm alternative.
Now we may rewrite $v_D$ in (\ref{VD}) as
\ben
v_D=q_{\kappa+2}+ C_{D,m}\,r^{\kappa+2}\left\{ \ln r Y_{\kappa+2}^{m}+\psi_{D,m} \right\},
\enn
where $q_{\kappa+2}=\zeta\,r^{\kappa+2}\,Y_n^{m}\in \PP_{\kappa+2}$ satisfies the equation $\Delta q_{\kappa+2}=p_\kappa$ (see (\ref{qkappa})) and
\ben
\psi_{D,m}=(W-\zeta Y_n^{m})/c, \quad C_{D,m}=c.
\enn
Hence we obtain the assertion for the Dirichlet boundary value problem with our special right hand side. The case of the Neumann boundary condition can be treated analogously.
\end{proof}

\subsection{Proof of Lemma \ref{lem} for circular cones}\label{sec:3-c}

 Recall that $B_1$ is the unit ball centered at the origin $O$
 and that $\CC \subset \R^3$ is an infinite circular cone with the angle $2\omega\in(0,2\pi)\backslash\{\pi\}$.
 Assume $q\in C^{\,l,s}(\overline{\CC \cap B_1})$ for some $l\in \N_0$, $s\in(0,1)$, satisfying $q\equiv1$ in $B_1\backslash\overline{\CC}$.
 Consider the coupling problem between the Helmholtz equations
\be\begin{split}\label{eq:1}
&\Delta u_1+ k^2 u_1=0,\qquad \Delta u_2+ k^2 q u_2=0 \quad \mbox{in}\quad B_1,\\
&\partial_\nu^{j}(u_1-u_2)=0\quad \mbox{on}\quad \partial \CC\cap B_1,\; j=1,2,\cdots,l+1,
\end{split}
\en
where $\partial_\nu^{j}$ denotes the normal derivative of order $j$ at $\partial \CC$ and $\nu$ is the unit normal pointing into the exterior of $\CC$.
The following lemma implies Lemma \ref{lem} in the case (iii) and the fact that a circular cone scatters each incident wave non-trivially. %Moreover, the proofs
% of Theorems \ref{TH1}, \ref{TH2} and Corollary \ref{Corollary} for a polygon follow straightforwardly from Lemma \ref{lem:corner-2D}; see Section \ref{sec: curvilinear} for the proof for curvilinear polygons.
\begin{lem}\label{lem-cone}
Let $u_1,u_2\in H^2(B_1)$ be a solution pair to (\ref{eq:1}), and suppose
 that $q$ satisfies the assumption (a) near the vertex $O$ with $D:=\CC\cap B_1$.
Then we have $u_1=u_2\equiv0$ in $B_1$.
\end{lem}

\begin{proof} We shall proceed  following the lines in the proof of Lemma \ref{lem:corner-2D}. In order to
avoid repeating the arguments used in Section \ref{sec:3}, we only indicate the necessary changes for circular cones. For simplicity we shall carry out the proof for H\"older continuous potentials only, i.e.,
 under the assumption (a) with $l=0$. Hence, we have $q\in C^{0,s}(\overline{\CC \cap B_1})$ and $q(O)\neq 1$.
The general case of $l\geq 1$ can be treated analogously to the proof of Lemma \ref{lem:corner-2D}.

\textbf{Step 1.} Choosing an appropriate cut-off function $\chi\in C_0^\infty(\overline{\CC})$ and
setting $v:=\chi (u_1-u_2)$, we have by Proposition \ref{P1} that $v\in C^{1,\alpha}(\overline{\CC})\cap H^{2}(\CC)$ for all $\alpha\in[0,1)$.
 Further,
\be\label{eq:2c}\begin{split}
&\Delta v=-k^2\chi\,qu+k^2\chi\,hu_1-[\Delta,\chi]v=:f\quad\mbox{in}\quad \CC,\\
&v=\partial_\nu v=0 \quad\mbox{on}\quad \partial \CC,
\end{split}
\en with  $h=1-q$. Here the commutator $[\cdot,\cdot]$ is defined in the same way as in  Section \ref{sec:3}.
 Applying Proposition \ref{P2C} with $\beta=-1/2$ and using the vanishing of the Cauchy data on $\partial \CC$, it follows that $v$ is the unique solution of (\ref{eq:2c}) in $V^2_{-1/2}(\CC)$.
 Note that we have $\lambda_{D/N,j}\neq 1$ for all $j\in \N$, because
\ben
&&P_1^0(\cos\omega)=\cos\omega\neq 0,\\
&&P_1^1(\cos\omega)=-(P_1^0)'(\cos\omega)=\sin\omega\neq 0,\\
&&(P_1^1)'(\cos\omega)=-\cos\omega/\sin\omega\neq 0
 \enn
 for all $\omega\in(0,\pi)\backslash\{\pi/2\}$.
 On the other hand, by Propositions \ref{P3C} and \ref{P4C}, there exist unique solutions $v_{D/N}$ of the first equation in (\ref{eq:2c}) satisfying  $v_{D/N} \in \Lambda_{\beta,D/N}^{2,\alpha}(\CC)$ for all $\beta\geq 1$ sufficiently close to 1 and $\alpha>0$ sufficiently small. Note that $2+\alpha-\beta\neq \lambda_{D/N,j}$ for those $\alpha,\beta$ and all $j\in \N$. Since $v_{D/N} \in \Lambda_{1}^{2,\alpha}(\CC)\cap \Lambda_{\beta}^{2,\alpha}(\CC)$ for some $\beta>1$,
it is easy to check that
 $v_{D/N}\in V^2_{-1/2}(\CC)$.
 Hence $v=v_D=v_N \in \Lambda_{1,D}^{2,\alpha}(\CC)\cap \Lambda_{1,N}^{2,\alpha}(\CC)$.

\textbf{Step 2.} To show that  $f\in \Lambda_{0}^{\,0,\alpha}(\CC),  v\in \Lambda_{0}^{2,\alpha}(\CC)$ and $u_1(O)=0$, we
rewrite the right hand side $f\in\Lambda_{1}^{0,\alpha}(\CC)$ in the form
\be\label{eq:5C}
f=\chi p_0+f_0,\quad p_0:=k^2\,h(O) \,u_1(O)\in \PP_0,\quad f_0:=f-\chi p_0\in \Lambda_{0}^{0,\alpha}(\CC),
\en
and consider the boundary value problems $\Delta_{D/N} v_0=p_0$ on $\overline{\CC}$.
 Applying Proposition \ref{Special-solution-C} with $\kappa=0$ yields special solutions $v_{D/N,0}$  of the form
 \be\label{eq:3C}\begin{split}
 v_{D,0}(x)&=&q_{D,2}(x)+c_D\,r^{ 2}\left\{ \ln r\, Y_2^{0}(\hat{x})+\psi_{D,0}(\hat{x}) \right\},\\
  v_{N,0}(x)&=&q_{N,2}(x)+c_N\,r^{ 2}\left\{ \ln r\, Y_2^{0}(\hat{x})+\psi_{N,0}(\hat{x}) \right\},
  \end{split}
 \en
 where $q_{D/N,2}\in \PP_{ 2}$ satisfy $\Delta q_{D/N,2}=p_0$, $\psi_{D/N,0}\in C^\infty(\overline{\Omega})$,
   $c_{D}=0$ if $P_2(\cos\omega)\neq 0$ and
 $c_{N}=0$ if $P_2'(\cos\omega)\neq 0$.
 Set $w_{D/N,0}:=v-\chi\, v_{D/N,0}\in \Lambda_1^{ 2,\alpha}.$
It follows from (\ref{eq:5C}) that
\ben
\Delta w_{D/N,0}=f_0-[\Delta,\chi]v_{D/N,0}\in \Lambda_0^{0,\alpha}(\CC)\cap \Lambda_1^{0,\alpha}(\CC).
\enn
 Applying Proposition \ref{P4C} with $\gamma_1=0$, $\gamma=1$ and $\alpha>0$ sufficiently small, we get the representations
  \be\label{eq:4C}
 w_{D/N,0}= \chi\,\sum_j d_{D/N,j}\, r^{ \lambda_{D/N,j}} \,V_{j,D/N}(\hat{x})  +\tilde{w}_{D/N}
  \en
  with $d_{D/N,j}\in\C$, $\tilde{w}_{D/N}\in \Lambda_{D/N,0}^{ 2,\alpha}(\CC)$,
  where $(\lambda_{D/N,j},V_{j,D/N})$ is the eigensystem corresponding to (\ref{eigenvalue}) and the sum is taken over all $j\in \N$ such that $\lambda_{D/N,j}\in( 1+\alpha,  2]$.  Here the eigenvalues are counted with their multiplicities. Note that we may assume that there are no Dirichlet and Neumann eigenvalues of  (\ref{eigenvalue}) in the interval $(2,2+\alpha)$.
Combining (\ref{eq:3C}) with (\ref{eq:4C}) and recalling that $v$ solves both the Dirichlet and Neumann boundary value problems, we obtain the following expressions for $v$ as $r\rightarrow 0$:
\be\label{eq:21C}\begin{split}
v&=&\sum_j d_{D,j}\, r^{ \lambda_{D,j}} \,V_{j,D}(\hat{x})
+q_{D,2}+ c_D\,r^{ 2}\left\{ \ln r\, Y_2^{0}+\psi_{D,0}    \right\}+\mathcal{O}(r^{2+\alpha})\\
&=&\sum_j d_{N,j}\, r^{\lambda_{N,j}} \,V_{j,N}(\hat{x})
+q_{N,2}+c_N\,r^{ 2}\left\{ \ln r\, Y_2^{0}+\psi_{N,0}   \right\}+\mathcal{O}(r^{2+\alpha}),
\end{split}
\en
from which we get the relations (see Step 3 below for the proof in the general case)
\ben
 c_{D/N}=0, \qquad\quad
d_{D/N,j}=0 \quad\mbox{if} \quad\lambda_{D/N,j}< 2.
\enn
Equating the lowest order terms in (\ref{eq:21C}) as $r\rightarrow 0$
allows us to define $q_2\in \PP_{ 2}$ as
\ben
q_2:=q_{D,2}+r^2\sum_{j\in\N: \lambda_{D,j}=2} d_{D,j} \,V_{j,D}(\hat{x})
=q_{N,2}+
r^2\sum_{j\in\N: \lambda_{N,j}=2} d_{N,j} \,V_{j,N}(\hat{x}).
\enn Using $\Delta r^{\lambda_{D/N,j}} \,V_{j,D/N}=0$ and Proposition \ref{P4}, we get
\ben
\Delta q_{2}=\Delta q_{D,2}=\Delta q_{N,2}=p_0\in\PP_0,\qquad \Delta^2 q_2=0\qquad\mbox{in}\quad \CC.
\enn Moreover, $q_2$ has  vanishing Cauchy data
$q_{2}=\partial_{\nu}q_{ 2}=0$ on $\partial \CC$.
Applying Proposition \ref{P6C} in the Appendix,
we arrive at $q_{ 2}\equiv0$, so that $p_0=0$. This implies that $v\in\Lambda_0^{ 2,\alpha}(\CC)$. Finally, the relation $u_1(O)=0$ follows from
 the definition of $p_0$ in (\ref{eq:5C}) and the assumption $q(O)\neq 1$.

\textbf{Step 3.} Assume for some $m>1$, $m\in \N$ and $\alpha>0$ sufficiently small that
\be\label{eq:6C}
 f\in \Lambda_{1-m}^{\,0,\alpha}(\CC),\quad  v\in \Lambda_{1-m}^{\, 2,\alpha}(\CC), \, \nabla^j u_1(O)=0\quad\mbox{for all}\; j\in \N_0,\; j\leq m-1.
\en We want to show in this step that
\be\label{eq:7}
 f\in \Lambda_{-m}^{\,0,\alpha}(\CC),\quad  v\in \Lambda_{-m}^{\, 2,\alpha}(\CC), \quad \nabla^m u_1(O)=0.
\en
Again denote by $u_{1,m}\in \PP_m$ the homogeneous Taylor polynomial of degree $m$ of $u_1$ at $O$.
By the last relation in (\ref{eq:6C}), we have  $u_{1,j}\equiv 0$ for all $j\leq m-1$. Using Proposition \ref{P7} (iii) we get $\Delta u_{1,m}\equiv0$.

By (\ref{eq:6C}), the right hand side in (\ref{eq:2c}) can be split into
\ben
f=\chi p_m+f_m,\quad p_m:=k^2\, h(O)\,u_{1,m}\in \PP_{ m},\quad f_m:=f-\chi p_m\in \Lambda_{-m}^{\,0,\alpha}(\CC).
\enn
Repeating the arguments in Step 2, we find that near the conic point $O$ the function $v\in \Lambda^{2,\alpha}_{1-m,D}(\CC)$ takes the form
\be\label{eq:23C}\begin{split}
v=q_{D, m+2}+c_{D,\kappa}\,r^{m+2}\left\{ \ln r\, Y_{m+2}^{\kappa}+\psi_{D,\kappa} \right\}
+\sum_j\,d_{D,j}\, r^{\lambda_{D, j}} \,V_{j,D}(\hat{x})+\tilde{w}_D
\end{split}
\en
 as a solution to the Dirichlet boundary value problem, whereas
 $v\in \Lambda^{2,\alpha}_{1-m,N}(\CC)$ can be expressed as
\be\label{eq:24C}\begin{split}
v=q_{N, m+2}+c_{N,\kappa'}\,r^{m+2}\left\{ \ln r Y_{m+2}^{\kappa'}+\psi_{N,\kappa'}\right\}
+\sum_j\,d_{N,j}\, r^{\lambda_{N, j}} \,V_{j,N}(\hat{x})+\tilde{w}_N
\end{split}
\en
 as a solution to the Neumann boundary value problem.
The parameters and functions involved in (\ref{eq:23C}) and (\ref{eq:24C}) are described as follows:
\begin{description}
\item[(i)] $\tilde{w}_{D/N}\in \Lambda^{2,\alpha}_{-m,D/N}(\CC)$, $\psi_{D,\kappa}, \psi_{N,\kappa'}\in C^\infty(\overline{\Omega})$. Hence $\tilde{w}_{D/N}=\mathcal{O}(r^{m+2+\alpha})$ as $r\rightarrow 0$.
\item[(ii)] The integers $\kappa$ and $\kappa'$ satisfy $|\kappa|, |\kappa'|\leq m$ and $P_{m+2}^{|\kappa|}(\cos\omega)= (P_{m+2}^{|\kappa'|})'(\cos\omega)=0$. Further, it holds that $|\kappa'|\neq |\kappa|$, since $P_{m+2}^{n}(\cos\omega)$ and $(P_{m+2}^{n})'(\cos\omega)$ cannot vanish simultaneously for $0\leq n\leq m+2$; see Proposition \ref{P8} (ii).
\item[(iii)] $c_{D/N}$, $d_{D/N,j}\in \C$. Moreover, $c_{D}=0$ if $P_{m+2}^{|\kappa|}(\cos\omega)\neq 0$ for all $|\kappa|\leq m$, while $c_{N}=0$ if $(P_{m+2}^{|\kappa'|})'(\cos\omega)\neq 0$ for all $|\kappa'|\leq m$.
  \item[(iv)]  The sums in (\ref{eq:23C}) and (\ref{eq:24C}) are taken over all $j\in \N$ such that the eigenvalues (counted with their multiplicities) fulfill $\lambda_{D/N,j}\in(m+1+\alpha, m+2]$.
   \item[(v)]  $q_{m+ 2,D/N}\in \PP_{m+2}$ satisfies  $\Delta q_{m+ 2,D/N}=p_m\in \PP_m$.
\end{description}
We first claim that
\be\label{eq:31}
d_{D/N,j}=0\quad\mbox{if}\quad \lambda_{D/N, j}<m+2.
\en
For this purpose we denote by $\lambda^*=\min_j\{\lambda_{D, j},\lambda_{N, j}\}$ the smallest exponent of $r$ on the right hand sides of (\ref{eq:23C}) and (\ref{eq:24C}).
Supposing on the contrary that (\ref{eq:31}) does not hold, we then have
 $\lambda^*<m+2$. Subtracting (\ref{eq:23C}) from (\ref{eq:24C}), multiplying $r^{-\lambda^*}$ to the resulting expression and letting $r\rightarrow0$, we arrive at $d_{D/N,j}=0$ for $\lambda_{D/N, j}=\lambda^*$ due to the orthogonality of the eigenfunctions $V_{j,D}$ and $V_{j,N}$. Repeating this process yields (\ref{eq:31}).

The relation (\ref{eq:31}) implies that $\lambda^*= m+2$. We now  multiply $(r^{m+2}\ln r)^{-1}$ to both equalities (\ref{eq:23C}) and (\ref{eq:24C}) and consider the difference of the resulting expressions to obtain $c_{D,\kappa}=c_{N,\kappa'}=0$, where we have used the linear independence of $P_{m+2}^{|\kappa|}$ and $P_{m+2}^{|\kappa'|}$ for $|\kappa|\neq |\kappa'|$.
Hence,
the lowest order term $q_{m+2}$ of $v$ near $O$ belongs to $\PP_{ m+2}$ and takes the form
\ben
q_{m+2}&=&q_{D, m+2}+\sum_{j:\lambda_{D, j}=m+2}\,d_{D,j}\, r^{m+2} \,V_{j,D}\\
&=&q_{N, m+2}+\sum_{j:\lambda_{N, j}=m+2}\,d_{N,j}\, r^{m+2} \,V_{j,N}.
\enn
This further yields
\ben
&\Delta q_{ m+2}=p_{m}\in\PP_{ m},\quad \Delta^2q_{m+2}=\Delta p_m=0\quad\mbox{in}\quad \CC\, \\
&q_{ m+2}=\partial_{\nu}q_{ m+2}=0\quad\mbox{on}\quad \partial \CC.
\enn
Again using Proposition \ref{P6C} in the Appendix, we get $q_{ m+2}\equiv0$. Consequently,  $p_{m}\equiv0$ and $u_{1,m}\equiv 0$, which implies the relations in (\ref{eq:7}).

\textbf{Step 4.} Having proved that $\nabla^j u_1(O)=0$ for all $j\in \N_0$ in the previous steps, we obtain $u_1\equiv0$ in $B_1$ due to the analyticity. Finally, the vanishing of $u_2$ follows from the unique continuation for the Helmholtz equation. This finishes the proof of Lemma \ref{lem-cone}.

\end{proof}

\section{Appendix}
%\appendix
\renewcommand\thesection{A}

%\section{Appendix}
 In the Appendix, we prove several propositions that are used in Sections \ref{sec:2D}-\ref{sec: curvilinear}.
    In particular, Propositions \ref{P5} and \ref{P7} below extend the results of \cite{BLS}. We present an alternative method
  of proof relying on the expansion of  real-analytic
   solutions, which is of independent interest.

\begin{prop}\label{P5} Suppose that $(\Delta+k^2)u=0$ in a neighbourhood of the point $O\in \R^2$. Then the two lowest order terms in the Taylor expansion of $u$ at $O$ are both harmonic functions.
\end{prop}
\begin{proof} Suppose that the lowest degree in the Taylor expansion of $u_1$ at $O$ is $M\in \N_0$ and that all terms of order less than $M$ vanish. Then the function $u_1=u_1(r,\theta)$ can be expanded into the convergent series (see, e.g., \cite[Lemma 2.2]{EHY2014})
\be\label{expan-2D}
u=\sum_{j\in\N_0,j\geq M}r^j F_j(\theta),\quad F_j(\theta)=\sum_{n,m\in \N_0,n+2m=j} \left(c^+_{n,m}\cos n\theta+c^-_{n,m}\sin n\theta\right),
\en
where $c^\pm_{n,m}\in\C$ satisfy the recurrence relations:
\ben
c^\pm_{n,m+1}=-\frac{k^2}{4(m+1)(n+m+1)}\,c^\pm_{n,m}\quad\mbox{for all}\quad n,m\in \N_0.
\enn In particular, the coefficients of the first three terms in the expansion are given by
\ben
&&F_0(\theta)=c_{0,0}^+,\\
&& F_1(\theta)=c_{1,0}^+\cos\theta+c_{1,0}^-\sin\theta,\\
&& F_2(\theta)=c_{0,1}^++c_{2,0}^+\cos2\theta+c_{2,0}^-\sin2\theta,\quad c_{0,1}^+=-c_{0,0}^+\,k^2/4.
\enn Hence, if $M=0$, it is obvious that both $F_0$ and $rF_1$ are harmonic. If $M=1$, we have $c_{0,0}^+=0$ and both $rF_1$ and $r^2F_2$ are harmonic functions. Now assume that $M\geq 2$. It then holds that $c^\pm_{n,m}=0$ for all $n+2m\leq M-1$.
For $n,m\in \N_0$, $m\geq 1$ such that $n+2m=M$, it follows from the recurrence relations that
$$c^\pm_{n,m}=-\frac{k^2}{4m(n+m)}c^\pm_{n,m-1}=0,$$
since $n+2(m-1)=M-2$. This implies that the lowest order term, given by
$$F_M=c_{M,0}^+\cos M\theta+ c_{M,0}^-\sin M\theta,$$
is harmonic. Analogously, one can prove that $r^{M+1}F_{M+1}(\theta)$ is also harmonic.
\end{proof}
Next we prove the result corresponding to Proposition \ref{P5} in 3D.
\begin{prop}\label{P7}
\begin{description}
\item[(i)] A real-analytic function $u=u(r,\theta,\varphi)$ can be expanded in a neighbourhood of the origin as follows:
\be\label{eq}
u(x)=\sum_{n,l\in \N_0} r^{n+2\,l}\,\sum_{m=-n}^n\,a_{n,m}^{(l)}\;Y_{n}^m(\theta,\varphi),\quad a_{n,m}^{(l)}\in \C.
\en
\item[(ii)] A solution to the Helmholtz equation $(\Delta+k^2)u=0$ can be expanded in the form (\ref{eq})
 where the coefficients $a_{n,m}^{(l)}$ fulfill the recurrence relations
    \ben
    a_{n,m}^{(l+1)}=-\frac{k^2}{2(l+1)(2l+2n+3)}\,a_{n,m}^{(l)},\quad n,l\in\N_0,\;m=-n,-n+1,\cdots,n-1,n.
    \enn
\item[(iii)] Suppose that $(\Delta+k^2)u=0$ in $\R^3$.
Then the two lowest order terms in the Taylor expansion of $u$ at $O\in \R^3$
are both harmonic functions in $\R^3$.
\end{description}
\end{prop}
\begin{proof}
(i) Recall that $\mathbb{P}_n$ denotes the collection of all homogeneous polynomials of degree $n\in\N_0$. We denote  by $\mathbb{H}_n$
the subset of $\mathbb{P}_n$ consisting of harmonic homogeneous polynomials of degree $n$. Then, for any $H_n\in\mathbb{H}_n$ there holds the expansion
\be\label{Eq:1}
H_n(x)= r^n \sum_{m=-n}^n\,c_{n,m}\; Y_n^m(\theta,\varphi),\quad c_{n, m}\in \C.
\en Since $\mathbb{P}_n=\mathbb{H}_n+|x|^2\; \mathbb{P}_{n-2}$, we obtain by induction that any $p_n\in \mathbb{P}_n$ can be written in the form
\be\label{eq:10}
p_n(x)=\sum_{l\in\N_0, n-2l\geq 0} b_l\; |x|^{2l}\,H_{n-2l}(x),\quad b_l\in\C,\;H_{n-2l}\in \mathbb{H}_{n-2l}.
\en
Since $u$ is real-analytic, applying the Taylor expansion and using (\ref{eq:10}) yields
\ben
u(x)=\sum_{n\in\N_0} c_n\,p_n(x)=\sum_{n\in\N_0} c_n\,\sum_{l\in\N_0, n-2l\geq 0} b_l\, |x|^{2l} \,H_{n-2l}(x).
\enn
Rearranging the terms in the previous expression, we get
\ben
u(x)=\sum_{l\in\N_0} |x|^{2l}\,\sum_{n\in \N_0} a_{n}^{(l)}\,H_{n}(x),\quad a_{n}^{(l)}\in \C.
\enn
which together with (\ref{Eq:1}) proves the first assertion.

(ii) The second assertion follows from the relation
\be\label{eq:32}\begin{split}
\Delta u&=& \sum_{n\in \N_0,l\in\N} 2l\,(2l+2n+1) r^{n+2\,l-2}\,\sum_{m=-n}^n\,a_{n,m}^{(l)}\;Y_{n}^m(\theta,\varphi)\\
&=&\sum_{n,l\in \N_0} 2(l+1)(2l+2n+3) r^{n+2\,l}\,\sum_{m=-n}^n\,a_{n,m}^{(l+1)}\;Y_{n}^m(\theta,\varphi).
\end{split}
\en
(iii) To prove the third assertion, we rewrite the expansion (\ref{eq}) as
\ben
u(x)=\sum_{j\in\N_0}r^j\,F_j(\theta,\varphi),\quad
F_j(\theta,\varphi):=\sum_{n,l\in\N_0, n+2l=j}\;\sum_{m=-n}^n\,a_{n,m}^{(l)}\;Y_{n}^m(\theta,\varphi).
\enn
Proceeding in the same way as in Proposition \ref{P5}, one can verify  the third assertion.
%{\color{MyDarkRed} To consider the second statement, we observe that the lowest degree of $u$ at $(O', t)$ coincides with that of $u(\cdot, t)$ at $O'$, which we denote by $N$.
%Let $u_{m}(\cdot, x_3)$, $m\in \N_0$ be the homogeneous Taylor expansion of degree $m$ of $u_1(\cdot,x_3)$ at $x'=O'$ (see (\ref{eq:17})). Then
%\be\label{eq:26}
%u(x',t)=u_N(x', t)+u_{N+1}(x', t)+\mathcal{O}(|x'|^{N+2})\quad\mbox{as}\quad |x'|\rightarrow 0.
%\en
%Analogously, we suppose as  $(x',x_3)$ tends to $(O',t)$ that
%\be\label{eq:22}\begin{split}
%&u(x',x_3)= U_N(x', x_3-t)+U_{N+1}(x',x_3-t)+\mathcal{O}((|x'|+|x_3-t|)^{N+2}),\\
%&U_m(x', x_3-t):=\sum_{|j|+l=m} C_{m,j,l}\; x'^{j}\;(x_3-t)^l,\quad m=N,N+1
%\end{split}
%\en for some $C_{m,j,l}\in \C$ with $j=(j_1,j_2)\in \N_0^2$, $|j|=j_1+j_2$ and $l\in \N_0$. Setting $x_3=t$ in (\ref{eq:22}) and comparing (\ref{eq:26}) and (\ref{eq:22}) we get
%\ben
%u_m(x', t)=\sum_{|j|=m} C_{m,j,0}\; x'^{j},\quad m=N,N+1.
%\enn Since the first term on the right hand side of (\ref{eq:22}) is harmonic at ($O',t$), we see
%\ben
% 0=[\Delta_x U_m(x', x_3-t)]|_{x_3=t}=\Delta_{x'} [u_m(x', t)] + 2\sum_{|j|=m-2} C_{m,j,2}\; x'^{j},\quad m=N,N+1.
%\enn This does not imply that $\Delta_{x'} [u_m(x', t)]\neq 0$ ???}
\end{proof}

In \cite{BLS},  Propositions \ref{P5} and \ref{P7} are verified for the lowest order term of solutions to the
  Helmholtz equation only. Proposition \ref{P6} below implies the absence of  non-trivial biharmonic functions with vanishing Dirichlet and Neumann data
 on the boundary of a sector in $\R^2$.

\begin{prop}\label{P6}
Let $\K=\K_\omega\subset\R^2$ be the sector defined in Section \ref{Results} with the opening angle $\omega\in(0,2\pi)\backslash\{\pi\}$. Suppose that
 $u\in H^2(B_1)$ solves the boundary value problem
\be\label{equation}
\Delta^2 u=0\quad\mbox{in}\quad \K,\qquad u=\partial_{\nu}u=0\quad\mbox{on}\quad\partial \K\cap B_1.
\en Then $u\equiv 0$.
\end{prop}
In \cite{KS}, Proposition \ref{P6} was proved for a homogeneous polynomial $p_l$
 such that $\Delta p_l$ is harmonic. Our proof differs from that in \cite{KS}. It is also elementary,
 since simple calculations
 using Cartesian coordinates are involved only. Alternatively, Proposition \ref{P6}  also follows from the expansion (\ref{expan-2D}) under polar coordinates; we refer to the proof of Proposition \ref{P6C} below where the spherical coordinates are employed to prove the analogue of Proposition \ref{P6} for circular cones in 3D.
\begin{proof}
Denote by $\tau_j$ and $\nu_j$ ($j=1,2$) the unit tangential and normal vectors on the two half-lines of $\partial \K$ starting at the corner $O$. Since the opening angle of $\K$ is not $\pi$, the tangential and normal vectors are linearly independent.  Without loss of generality we suppose that $\nu_1=c_1\tau_1+c_2\tau_2$ with $c_1,c_2\in\R$, $c_2\neq 0$. Hence,
\be\label{Derivative}
\partial_{\tau_2}=\frac{1}{c_2}\partial_{\nu_1}-\frac{c_1}{c_2}\partial_{\tau_1}.
\en
We shall prove by induction  that $\nabla^m u(O)=0$ for all $m\in\N_0$, which implies the proposition.

From the Dirichlet and Neumann boundary conditions of $u$ on $\partial \K$ we see that
\be\label{D-12}
u=\nabla u=0,\quad
\partial_{\tau_1}^2u=\partial_{\tau_2}^2u=\partial_{\nu_1}\partial_{\tau_1}u=0\quad\mbox{at the corner}\; O.
\en
Combining (\ref{Derivative}) and (\ref{D-12}) gives the relation
$\partial_{\tau_1}\partial_{\tau_2}u=0$ at $O$.
Since each entry of the vector $\nabla^2$ can be expanded as a linear combination of $\partial_{\tau_1}^2$, $\partial^2_{\tau_2}$ and $\partial_{\tau_1}\partial_{\tau_2}$, we obtain $\nabla^2u=0$ at $O$.

To prove that $\nabla^3u(O)=0$, we observe that
\ben
\partial_{\tau_1}^3u=\partial_{\tau_2}^3u=
\partial_{\tau_1}^2\partial_{\nu_1}u=
\partial_{\tau_2}^2\partial_{\nu_2}u=0 \quad\mbox{at}\; O.
\enn
Applying $\partial_{\tau_1}^2$ to both sides of (\ref{D-12}) yields $\partial_{\tau_1}^2\partial_{\tau_2}u(O)=0$. Analogously we can get $\partial_{\tau_2}^2\partial_{\tau_1}u(O)=0$.
Hence, the relation $\nabla^3u(O)=0$ follows from the fact that the differential operators $\partial_{\tau_1}^3, \partial_{\tau_1}^2\partial_{\tau_2}, \partial_{\tau_1}\partial_{\tau_2}^2$ and $\partial_{\tau_2}^3$ span the vector $\nabla^3$.

Now we want to verify that $\nabla^4u(O)=0$. Arguing as in the previous step we get
\be\label{eq:8}
\partial_{\tau_1}^4u=\partial_{\tau_1}^3\partial_{\tau_2}u=\partial_{\tau_1}\partial_{\tau_2}^3u=
\partial_{\tau_2}^4u=0\quad\mbox{at}\; O.
\en Hence it suffices to prove $\partial_{\tau_1}^2\partial_{\tau_2}^2u(O)=0$. Using (\ref{eq:8}), $\partial_{\nu_1}=c_1\partial_{\tau_1}+c_2\partial_{\tau_2}$ and $\Delta^2u\equiv0$, this follows from the identity
\ben
0=\Delta^2u(O)=[\partial_{\nu_1}^2 +\partial_{\tau_1}^2]^2u(O)=[2(1+c_1^2)c_2^2+4c_1^2c_2^2]\,\partial_{\tau_1}^2\partial_{\tau_2}^2u(O).
\enn
For $m>4$, we make the induction hypothesis that
\be\label{induction-hypothesis}
\nabla^j u(O)=0\quad\mbox{for all}\;j=0,1,\cdots,m.
\en We then only need to verify that $\nabla^{m+1}u =0$ at $O$.
For $j\in \N_0$, denote by $\nabla^j_\tau$ the vector of all tangential derivatives of order $j$, i.e.,
\ben
\nabla^j_\tau u=\left\{\partial_{\tau_1}^{j_1}\partial^{j_1}_{\tau_2}\,u: \quad j_1,j_2\in\N_0,j_1+j_2=j\right\}.
\enn Using the relations in (\ref{equation}) and (\ref{Derivative}) again, we have
\ben
\nabla_\tau^{m-3}\Delta^2u=\partial_{\tau_1}^{m+1}u=\partial_{\tau_1}^m\partial_{\tau_2}u
=\partial_{\tau_1}\partial_{\tau_2}^mu=\partial_{\tau_2}^{m+1}u=0\quad\mbox{at}\quad O.
\enn
Therefore, it remains to prove that the span of the differential operators $\nabla_\tau^{m-3}\Delta^2$,
$\partial_{\tau_1}^{m+1}$, $\partial_{\tau_1}^m\partial_{\tau_2}$,
$\partial_{\tau_1}\partial_{\tau_2}^m$ and $\partial_{\tau_2}^{m+1}$ contains the vector $\nabla^{m+1}_\tau$.

It can be readily checked that
\ben
\Delta&=&(1+c_1^2)\Lambda_1(\partial)\Lambda_2(\partial),\\
 \partial_{\tau_1}&=&-\frac{1}{2ic\,\ima c}(\Lambda_1(\partial)+\zeta\Lambda_2(\partial)),\\ \partial_{\tau_2}&=&\frac{1}{2i\,\ima c}(\Lambda_1(\partial)-\Lambda_2(\partial)),
\enn where
\ben
&&c:=\frac{c_1c_2+ic_2}{1+c_1^2},\quad \zeta:=-c/\overline{c},\\
&&\Lambda_1(\partial):=\partial_{\tau_1}+c\,\partial_{\tau_2},\quad
\Lambda_2(\partial):=\partial_{\tau_1}+\overline{c}\,\partial_{\tau_2}.
\enn Consequently, it suffices to verify that the span of the differential operators $$\Lambda_1^{j_1}\Lambda_2^{j_2}\Lambda_1^2\Lambda_2^2, \qquad\forall\; j_1,j_2\in\N_0,\, j_1+j_2=m-3,$$
 together with
 $$(\Lambda_1+\zeta\Lambda_2)^{m+1},\; (\Lambda_1+\zeta\Lambda_2)^m(\Lambda_1-\Lambda_2),\;
 (\Lambda_1+\zeta\Lambda_2)(\Lambda_1-\Lambda_2)^m,\; (\Lambda_1-\Lambda_2)^{m+1}$$
 contains the set of differential operators
 $\{\Lambda_1^{j_1}\Lambda_2^{j_2}:j_1+j_2=m+1\}$. This is equivalent to the claim that the polynomial expressions containing the terms $z_1^{m+1}, z_1^m z_2, z_1z_2^m, z_2^{m+1}$ in the expansion of
 \ben
 (z_1-z_2)^m z_1,\; (z_1-z_2)^m z_2,\; (z_1+\zeta z_2)^m z_1,\; (z_1+\zeta z_2)^m z_2
 \enn are linearly independent.
 Simple calculations show that
 \ben\begin{pmatrix}
 (z_1-z_2)^m z_1\\ (z_1-z_2)^m z_2\\
  (z_1+\zeta z_2)^m z_1\\ (z_1+\zeta z_2)^m z_2
 \end{pmatrix}=
 \begin{pmatrix}
 1 & -m & (-1)^m & 0\\
 0 & 1 & (-1)^{m-1}m & (-1)^m \\
 1 & m \zeta & \zeta^m  &0\\
 0 & 1 & m\zeta^{m-1} & \zeta^m
 \end{pmatrix}
 \begin{pmatrix}
 z_1^{m+1}\\
 z_1^{m}z_2\\
 z_1z_2^{m}\\
z_2^{m+1}
 \end{pmatrix} +\sum_{j=2 }^{m-1} M_j(\xi) z_1^{j}z_2^{m+1-j}
 \enn with $M_j\in \R^{4\times 1}$.
  It is easy to check that the determinant of the 4-by-4 coefficient matrix on the left hand side of the previous equation vanishes if and only if
 \ben
m^2 \zeta^{m-1} (1+\zeta)^2+(-1)^m [(-1)^m\zeta^m-1]^2=0.
 \enn
If $m\in\N$ is an odd number, the previous relation implies that
\ben
m^2 \zeta^{m-1}=(\frac{1+\zeta^m}{1+\zeta})^2=(\zeta^{m-1}-\zeta^{m-2}+\cdots+1)^2.
\enn
Since $|\zeta|=1$, the modulus of the right hand side of the previous identity equals to $m^2$ only if $\zeta=-1$, which however is impossible.  If $m$ is even, the number $\zeta_1=-\zeta$ is a solution of
\ben
-m^2 \zeta_1^{m-1}=-(\frac{1-\zeta_1^m}{1-\zeta_1})^2=-(\zeta_1^{m-1}+\zeta_1^{m-2}+\cdots+1)^2,
\enn which cannot hold for $|\zeta_1|=1$ and $\zeta_1\neq 1$.
\end{proof}

\begin{prop}\label{P6C}
Let $\CC=\CC_\omega\subset\R^3$ be the circular cone defined by (\ref{cone}) with the opening angle $2\omega\in(0,2\pi)\backslash\{\pi\}$. Suppose that
 $u\in H^2(B_1)$ solves the boundary value problem
\be\label{equationC}
\Delta^2 u=0\quad\mbox{in}\quad \CC\cap B_1,\qquad u=\partial_{\nu}u=0\quad\mbox{on}\quad\partial \CC\cap B_1.
\en Then $u\equiv 0$.
\end{prop}
Proposition \ref{P6C} extends the result of Proposition \ref{P6} in a planar corner domain to a circular conic domain in $\R^3$.
Being different from the proof of Proposition \ref{P6} using Cartesian coordinate, our proof of Proposition \ref{P6C} relies on the expansion of real-analytic functions using the spherical coordinates.
\begin{proof}
By Proposition \ref{P7} (ii), a real-analytic function $u=u(r,\theta,\varphi)$ in $B_1$ can be expanded  as the following convergent series
\be\label{eq:9}
u(x)=\sum_{n,l\in \N_0} r^{n+2\,l}\,\sum_{m=-n}^n\,a_{n,m}^{(l)}\;Y_{n}^m(\theta,\varphi),\quad a_{n,m}^{(l)}\in \C.
\en
Simple calculations using  (\ref{eq:32})  shows that
\ben
0=\Delta^2 u
=\sum_{n,l\in \N_0} 4(l+1)(l+2)(2l+2n+3)(2l+2n+5) r^{n+2\,l}\,\sum_{m=-n}^n\,a_{n,m}^{(l+2)}\;Y_{n}^m(\theta,\varphi).
\enn The previous relation implies that $a_{n,m}^{(l+2)}=0$ for all $l, n\in \N_0$ and $|m|\leq n$, since $r^{n+2\,l}\,Y_{n}^m \in \PP_{n+2l}$ are linearly independent. Hence, we only need to prove that $a_{n,m}^{(l)}=0$ for all $l=0,1$ and $n\in \N_0$, $|m|\leq n$.

The expansion of $u$ in (\ref{eq:9}) can be rewritten as
\be\label{u(x)}
u(x)=\sum_{n\in \N_0, l=0,1} r^{n+2\,l}\,\sum_{m=-n}^n\,a_{n,m}^{(l)}\;Y_{n}^m(\hat{x})=:\sum_{n\in \N_0} r^{n}F_n(\hat{x}),\quad
\hat{x}=(\theta,\varphi),
\en
with
\ben
F_n(\hat{x}):=
\left\{\begin{array}{lll}
a_{0,0}^{(0)}Y_{0}^0(\hat{x})\;&&\mbox{if}\quad n=0;\\
\sum_{m=-1}^1 a_{1,m}^{(0)}\;Y_{1}^m(\hat{x})&&\mbox{if}\quad n=1;\\
\sum_{m=-n}^n\,a_{n,m}^{(0)}\;Y_{n}^m(\hat{x})+\sum_{m=-n+2}^{n-2} a_{n-2,m}^{(1)}\;Y_{n-2}^m(\hat{x})&&\mbox{if}\quad n\geq 2.
\end{array}\right.
\enn
Making use of the boundary conditions
\ben
u=\partial_\theta u=0\quad\mbox{on}\quad \{(r,\theta,\varphi): 0<r<1,\,\theta=\omega,\,0\leq\varphi<2\pi\},
\enn we see that $F_n(\omega,\varphi)=\partial_\theta F_n(\omega,\varphi)=0$ for all $0\leq\varphi<2\pi$.
In view of the definition of the spherical harmonics (see (\ref{spherical-harmonics})), we obtain the following results
by inserting (\ref{u(x)}) into the boundary conditions and equating the coefficients of equal powers of $r$ :
\begin{description}
\item[(i)] $a_{0,0}^{(0)}=0$ in the case $n=0$, because $Y_{0}^0\equiv\sqrt{1/(2\pi)}\neq 0$.
\item[(ii)] $a_{1,m}^{(0)}P_1^{|m|}(\cos\omega)=a_{1,m}^{(0)}(P_1^{|m|})'(\cos\omega)=0$ for $m=-1,0,1$ when $n=1$.
Applying Proposition \ref{P8} (ii), it follows that $a_{1,m}^{(0)}=0$, since $P_n^{|m|}(t)$ and $(P_n^{|m|})'(t)$ cannot vanish simultaneously for any $t\in (-1,1)$.
\item[(iii)] For all $n\geq 2$ and $|m|\leq n-2$,
\be\label{eq:33}
\begin{pmatrix}
P_n^{|m|}(\cos\omega)& P_{n-2}^{|m|}(\cos\omega)\\
(P_n^{|m|})'(\cos\omega)& (P_{n-2}^{|m|})'(\cos\omega)
\end{pmatrix}
\begin{pmatrix}
a_{n,m}^{(0)}\\
a_{n-2,m}^{(1)}
\end{pmatrix}=0.
\en
By Proposition \ref{P8} (i) below, the determinant of the matrix on the left hand side of (\ref{eq:33}) never vanishes for $\omega\in (0,\pi)\backslash\{\pi/2\}$. Therefore, $a_{n,m}^{(0)}=a_{n-2,m}^{(1)}=0$ for $n\geq 2$ and $|m|\leq n-2$.
\item[(iv)] For all  $n\geq 2$ and $|m|=n,n-1$,
\ben
a_{n,m}^{(0)}P_n^{|m|}(\cos\omega)=a_{n,m}^{(0)}(P_n^{|m|})'(\cos\omega)=0.
\enn
In view of Proposition \ref{P8} (ii) we get $a_{n,m}^{(0)}=0$  for all $n \geq 2$ and $|m|=n,n-1$.
\end{description}

To sum up the above results in (i)-(iv), we obtain $a_{n,m}^{(l)}=0$  for all $l=0,1$, $n\in \N_0$ and $|m|\leq n$, $m\in \Z$, which finishes the proof of the proposition.

\end{proof}

\begin{prop}\label{P8}
Let $t\in(-1,1)$ and $m, n\in \Z$.
\begin{itemize}
\item[(i)] It holds that
\be\label{Det}
\mbox{det}\begin{pmatrix}
P_n^{m}(t)& P_{n-2}^{m}(t)\\
(P_n^{m})'(t)& (P_{n-2}^{m})'(t)
\end{pmatrix}\neq 0\qquad\mbox{for all}\quad t\neq 0,\, n-2\geq m\geq 0.
\en
\item[(ii)] It cannot happen that $P_n^{m}(t)=(P_n^{m})'(t)=0$ for all $0\leq m\leq n$.
\end{itemize}
 \end{prop}
\begin{proof}
(i) Introduce the augmented Wronskian of the form
\ben
W_n(t;j)=\mbox{det}\begin{pmatrix}
P_n(t)& P_{n-j}(t)\\
P_n'(t)& P_{n-j}'(t)
\end{pmatrix},\quad j=1,2,\cdots,n.
\enn
The number $t_0\in (-1,1)$ is called a nodal zero of $W_n$ if $W_n$ has opposite signs for  $t=t_0+h$ and  $t=t_0-h$, $h$ sufficiently small. It has been shown in \cite[Chapter 4, Theorem 9]{KS1961} that $W_n(t;j)$ has exactly $j-1$ nodal zeros in the interval $(-1,1)$. Hence, when $j=2$, $W_n(t;2)$ has only one nodal zero $t_0$ in  $(-1,1)$. If $n\geq 2$ is odd, then $P_n(0)=P_{n-2}(0)=0$, since both $P_n$ and $P_{n-2}$ are odd functions. This implies that
 $t_0=0$ is the nodal zero of  $W_n(t;2)$. If $n\geq 2$ is even, we have $P_n'(0)=P'_{n-2}(0)=0$.
Hence, $t_0=0$ is also the nodal zero. This proves the first assertion with $m=0$.

In the case  $m\geq 1$, the functions
 $P_n^m(t), n=m,m+1,\cdots,$  satisfy the associated Legendre differential equation (\ref{Legendre-A}). %Moreover, we have (see e.g., \cite[Chapter 2]{CK})
% \ben
% \mbox{Span}\{P_m^m, P_{m+1}^m,\cdots, P_n^m\}=(1-t^2)^{m/2}\mbox{Span}\{P_1,P_2,\cdots, P_{n-m}\}.
 %\enn
The proof of \cite[Chapter 4, Theorem 9]{KS1961} depends solely on the form of the governing equation (see (\ref{Legendre}) in the case of Legendre polynomials) and extends to the associated Legendre differential equation (\ref{Legendre-A}). Hence,
the determinant on the right hand side of (\ref{Det}) has also one nodal zero in $(-1,1)$. On the other hand,
it is easy to check that either $P_n^m(0)=P_{n-2}^m(0)=0$ or $(P^m_n)'(0)=(P^m_{n-2})'(0)=0$, implying that $t_0=0$ is the unique nodal zero. Hence, the first assertion for $m\geq 1$ follows from the proof for the Legendre polynomials.

(ii) The second assertion is a consequence of the fact that the zeros of $P_n^{|m|}$ and $(P_n^{|m|})'$ are all simple and strictly interlaced. Note that when $|m|=n$, we have the explicit expression (see e.g. \cite[Chapter 2.4]{Nedelec})
\ben
P_n^{n}(t)=\frac{(2n)!}{2^n n!}(1-t^2)^{n/2}.
\enn
\end{proof}

Finally we present a corollary that extends the results of
Propositions \ref{P6} and \ref{P6C} to a more general case. It can also be considered as
a local non-solvability result on the Cauchy problem for the Laplace equation on a cone and it is proved just as
Lemmas \ref{lem:corner-2D} and \ref{lem-cone}.
\begin{cor}\label{cor}
Let $\U$ be the sector $\K_\omega\subset \R^2$ or the cone $\CC_{\omega/2}\subset \R^3$  defined in Section \ref{Results} with the opening angle $\omega\in(0,2\pi)\backslash\{\pi\}$. Suppose that
 $u\in H^2(\U\cap B_1)$ solves the Cauchy problem
\ben
\Delta u=h\,g\quad\mbox{in}\quad \U,\qquad u=\partial_{\nu}u=0\quad\mbox{on}\quad\partial \U\cap B_1,
\enn
where $h\in C^\alpha(\overline{\U\cap B_1})$ for some $\alpha\in(0,1)$, $h(O)\neq 0$ and $(\Delta+\lambda)g=0$ in $B_1$ for some $\lambda\in \C$.
Then $u\equiv 0$.
\end{cor}
Note that Corollary \ref{cor} does not hold in the case of a half space ($\omega=\pi$) where even a global existence result for the analytic Cauchy problem can be proved; see \cite[Theorem 9.4.8]{Hor}.

%\section*{Acknowledgment}
%This work was finished while G. Hu worked at Weierstrass Institute (WIAS). He gratefully acknowledges the support provided by the institute and the research group 4 "Nonlinear Optimization and Inverse Problems".


\begin{thebibliography}{99}



\bibitem{Rondi05} G. Alessandrini  and L. Rondi,
 Determining a sound-soft polyhedral scatterer by a single far-field measurement,
 Proc. Amer. Math. Soc.,  133 (2005): 1685--1691 (Corrigendum: arXiv: math/0601406v1, 2006).

 \bibitem{BLS} E. Bl\aa{}sten, L. P\"aiv\"arinta and J. Sylvester, Corners always scatter, Commun. Math.
Phys., 331 (2014): 725-753.

\bibitem{B} A. L. Bukhgeim,  Recovering a potential from Cauchy data in the two dimensional case, J. Inverse and Ill-posed Problem,  16 (2008): 19--33.


\bibitem{CGH2010} F. Cakoni, D. Gintides and H. Haddar,  The existence of
an infinite discrete set of transmission eigenvalues, SIAM J. Math. Anal. , 42 (2010): 237-255.




 \bibitem{CY} J. Cheng and M. Yamamoto,
 Uniqueness in an inverse scattering problem with non-trapping polygonal obstacles with at most two incoming waves,
Inverse Problems, 19 (2003): 1361-1384 (Corrigendum: Inverse Problems, 21 (2005): 1193).

\bibitem{CK} D. Colton and R. Kress, {\em Inverse Acoustic and Electromagnetic Scattering Theory}, Springer, New York, 1998.

\bibitem{CP} D. Colton and P. Monk, The inverse scattering problem for timeharmonic acoustic waves in an inhomogeneous medium, Quart. J. Mech. Appl. Math., 41 (1988): 97-125.

 \bibitem{CPJ} D. Colton,  L. P\"aiv\"arinta and J. Sylvester, The interior transmission problem, Inverse Probl. Imaging, 1 (2007): 13-28.

\bibitem{CS1983} D. Colton  and
B. D. Sleeman, Uniqueness theorems for the inverse problem of acoustic scattering, IMA J. Appl. Math., 31 (1983): 253--259.

\bibitem{DP} M. Dauge and M. Pogu,
Existence et r\'{e}gularit\'{e} de la fonction potentiel pour des \'{e}coulements de fluides parfaits s'\'{e}tablissant autour de corps \`{a} singularit\'{e} conique, Ann. Fac. Sci. Toulouse Math. (5), 9 (1988): 213-242.

\bibitem{EY06} J.  Elschner  and M. Yamamoto,
 Uniqueness in determining polygonal sound-hard obstacles with a single incoming wave,
Inverse Problems, 22 (2006): 355-364.


%\bibitem{EY08} J.  Elschner  and M. Yamamoto,
 %Uniqueness in determining polyhedral sound-hard obstacles with a single incoming wave,
%Inverse Problems, 24 (2008): 035004.

\bibitem{ElHu2015} J. Elschner  and G. Hu, Corners and edges always scatter, Inverse Problems, 31 (2015): 015003.


\bibitem{ElHu} J. Elschner  and G. Hu,  Uniqueness in inverse transmission scattering problems for multilayered obstacles,
 Inverse Problems and Imaging,  5 (2011): 793-813.

\bibitem{EHY2014} J. Elschner, G. Hu and M. Yamamoto, Uniqueness in inverse elastic scattering from unbounded rigid surfaces of rectangular type, Inverse Problems and Imaging, 9 (2015): 127-141.

\bibitem{GT} D. Gilbarg and N. S. Trudinger,
Elliptic Partial Differential Equations of Second Order,
Springer, Berlin, 2001.

\bibitem{Gintides} D. Gintides, Local uniqueness for the inverse scattering problem in acoustics via the Faber- Krahn inequality, Inverse Problems, 21 (2005): 1195-1205.

\bibitem{Grisvard} P. Grisvard, {\em Singularities in Boundary Value Problems},
Masson, Paris, 1992.

\bibitem{Hor} L. H\"{o}rmander,  The Analysis of Linear Partial Differential Operators I, Springer, Berlin, 1983.

\bibitem{HSV} G. Hu, M. Salo, E. V. Vesalainen, Shape identification in inverse medium scattering problems with a single far-field pattern, SIAM J. Math. Anal.,  48 (2016): 152-165.

\bibitem{HL14} G. Hu and X. Liu, Unique determination of balls and polyhedral scatterers with a single point source wave, Inverse Problems, 30 (2014): 065010

\bibitem{OYa} O. Y. Imanuvilov and M. Yamamoto,  Inverse boundary value problem for Helmholtz equation in two dimensions,  SIAM J. Math. Anal., 44 (2012):  1333--1339.

\bibitem{Isakov90} V. Isakov, On uniqueness in the inverse transmission scattering problem,  Comm. Part. Diff. Equat.,  15 (1990):  1565-1587.

 \bibitem{Isakovbook} V. Isakov, Inverse Problems for Partial Differential Equations, Springer-Verlag, New York, 1998.




\bibitem{KS1961} S. Karlin, G. Szeg\"o, On certain determinants whose elements are orthogonal polynomials, Journal
d' Analyse Math., 8 (1961): 1-157.

 \bibitem{Kirsch93} A. Kirsch and R. Kress,  Uniqueness in inverse obstacle scattering,  Inverse Problems,
9 (1993): 285-299.

\bibitem{Kondratiev} V. A. Kondratiev, Boundary value problems for elliptic equations in domains with conical or angular points, Trans. Moscow Math. Soc., 16 (1967): 227--313.

\bibitem{KMR} V. A. Kozlov, V. G. Maz'ya and J. Rossmann, Elliptic Boundary Value Problems in Domains with Point Singularities, American Mathematical Society, Providence, 1997.

\bibitem{KS} S. Kusiak and J. Sylvester, The scattering support, Communications on Pure and Applied Mathematics, 56 (2003): 1525-1548.

\bibitem{Lax} P.D. Lax and R.S. Phillips, Scattering Theory, Academic Press, New York, 1967.

\bibitem{LHZ06} H. Liu   and J. Zou,
Uniqueness in an inverse obstacle scattering problem for both sound-hard and sound-soft polyhedral scatterers,
Inverse Problems,  22 (2006): 515-524.

\bibitem{MNP} V. G. Maz'ya, S. A. Nazarov and B. A. Plamenevskii,
{\em Asymptotic Theory of Elliptic Boundary Value Problems in Singularly Perturbed Domains I}, Birkh\"auser-Verlag, Basel, 2000.



 \bibitem{Nac} A. Nachman,  Reconstruction from boundary measurements, Ann. of Math.,  128 (1988): 531--576.

\bibitem{NP} S. A. Nazarov and B. A. Plamenevsky, Elliptic Problems in Domains with Piecewise Smooth Boundaries, Walter de Gruyter, Berlin, 1994.

\bibitem{Nedelec} J. C. Nedelec, Acoustic and Electromagnetic Equations. Integral Representations for Harmonic Problems, Springer-Verlag, New York, 2000.

\bibitem{PSV} L. P\"aiv\"arinta, M. Salo and E. V. Vesalainen, Strictly convex corners scatter, arXiv:1404.2513, 2014.

 \bibitem{SG} J. Sylvester and G. Uhlmann, A global uniqueness theorem for an
inverse boundary value problem,  Ann. of Math., 125 (1987): 153--169.

\bibitem{SU2004} P. Stefanov and G. Uhlmann, Local uniqueness for the fixed energy fixed angle inverse problem in obstacle scattering, Proc. Amer. Math. Soc, 132 (2004): 1351--1354.

\bibitem{S} J. Sylvester,  Discreteness of transmission eigenvalues via upper triangular compact operators, SIAM J. Math. Anal.,  44 (2012): 341--354.






\end{thebibliography}
\end{document}